\documentclass[a4paper,11pt,british]{article}
\usepackage{setspace,graphicx,
epstopdf,amsmath,amsfonts,amsgen, mathtools,
amstext,amsthm,amsbsy,amsopn,amssymb,
bbm,tikz, bbm,
parskip,verbatim,mathrsfs,enumerate,
xcolor,comment}  
\usepackage[utf8]{inputenc}   
\usepackage[top=2.5cm, bottom=3cm, left=2.0cm, right=2.0 cm]{geometry}
\usepackage[square,numbers]{natbib} 
\usepackage{todonotes}
\usepackage[pdfborder={0 0 0}]{hyperref}

\def\b0{\boldsymbol{0}}

\newcommand{\R}     {\mathbb{R}} 
\newcommand{\Z}     {\mathbb{Z}} 
 
\renewcommand{\P}   {\mathbb{P}}

\renewcommand{\S}     {\mathbb{S}}

\newcommand{\Hcal}   {{\mathcal H }}

\newcommand{\Vcal}   {{\mathcal V }} 
\newcommand{\Wcal}   {{\mathcal W }}

\newcommand{\Exp}{\mathscr{E}\kern-0.2mm{\operatorname{xp}}}
\newcommand{\Log}{\mathscr{L}\kern-0.2mm{\operatorname{og}}}

\def\1{{\mathchoice {1\mskip-4mu\mathrm l}      
{1\mskip-4mu\mathrm l} 
{1\mskip-4.5mu\mathrm l} {1\mskip-5mu\mathrm l}}}

\usepackage{mathtools}

\numberwithin{equation}{section}
\numberwithin{figure}{section}
\newtheoremstyle{plain}
  {6pt}
  {4pt}
  {\slshape}
  {}
  {\bfseries}
  {.}
  {0.5em}
  {}%
\newtheorem{thm}{\protect\theoremname}
  \newtheorem{defn}[thm]{\protect\definitionname}
  
  \newtheorem{prop}[thm]{\protect\propositionname}
  \newtheorem{rem}[thm]{\protect\remarkname}
  
  \newtheorem{lem}[thm]{\protect\lemmaname}
  \numberwithin{thm}{section}

\usepackage[english]{babel}
  \addto\captionsbritish{\renewcommand{\corollaryname}{Corollary}}
  \addto\captionsbritish{\renewcommand{\definitionname}{Definition}}
  \addto\captionsbritish{\renewcommand{\factname}{Fact}}
  \addto\captionsbritish{\renewcommand{\propositionname}{Proposition}}
  \addto\captionsbritish{\renewcommand{\remarkname}{Remark}}
  \addto\captionsbritish{\renewcommand{\theoremname}{Theorem}}
  \addto\captionsenglish{\renewcommand{\corollaryname}{Corollary}}
  \addto\captionsenglish{\renewcommand{\definitionname}{Definition}}
  \addto\captionsenglish{\renewcommand{\factname}{Fact}}
  \addto\captionsenglish{\renewcommand{\propositionname}{Proposition}}
  \addto\captionsenglish{\renewcommand{\remarkname}{Remark}}
  \addto\captionsenglish{\renewcommand{\theoremname}{Theorem}}
  \providecommand{\corollaryname}{Corollary}
  \providecommand{\definitionname}{Definition}
  \providecommand{\factname}{Fact}
  \providecommand{\propositionname}{Proposition}
  \providecommand{\remarkname}{Remark}
\providecommand{\theoremname}{Theorem}
\providecommand{\lemmaname}{Lemma}

\usepackage{authblk}
\title{Mermin--Wagner theorem \\ for dimers, monomer double-dimers, \\ and spatial random permutations}
\author[1]{Lorenzo Taggi}
\author[2]{Wei Wu}
\affil[1]{\small{Sapienza Universit\`a di Roma, Dipartimento di Matematica, Roma, Italy.}}
\affil[2]{\small{NYU Shanghai,  Mathematics Department and NYU-ECNU Math Institute, Shanghai, China.}}
\date{\today}

\begin{document}

\maketitle

\begin{abstract}
We study a generalisation of the double-dimer model that encompasses several models of interest, including the monomer double-dimer model, spatial random permutations, the dimer model, and the spin $O(N)$ model, and which is also related to the loop $O(N)$ model. 
We show that on two-dimensional-like graphs (such as slabs), both the correlation function and the probability that a loop visits two vertices decay to zero as the distance between the vertices diverges. 
Our approach is based on the introduction of a new complex spin representation for all models in this class, together with a new proof of the Mermin--Wagner theorem that does not require positivity of the Gibbs measure. 
Even for the well-studied dimer and double-dimer models our results are new: since they do not rely on exact solvability or Kasteleyn’s theorem, they apply beyond the planar-graph setting. 
\end{abstract}

\section{Introduction}

We study a general class of statistical mechanics models, naturally described in terms of random subgraphs or collections of interacting random walks on two-dimensional (not necessarily planar) graphs. This framework encompasses three paradigmatic examples: the \textit{dimer model,}  \textit{spatial random permutations,} and the \textit{monomer double-dimer model}. 
The main result is a \textit{Mermin–Wagner theorem} for the whole class, from which the decay of correlations and absence of long-range order or of macroscopic loops on two-dimensional graphs follow,  and which is new for all these models.

\subsection{Main results}
In what follows, we provide a brief description of these models and discuss the novelty of our results in each case.

\paragraph{The dimer model.}
The dimer model is the study of the set of perfect matchings of a graph.
It has attracted interest from a wide range of perspectives, including probability, combinatorics, statistical mechanics, and graph theory.
Configurations of the model are subgraphs of a fixed (bipartite) graph in which every vertex has degree exactly one, and they are sampled uniformly at random.

A major breakthrough in the rigorous study of such models was achieved through the seminal works of Kasteleyn \cite{Kasteleyn}, Temperley and Fisher \cite{Temperley}, who established a theorem (now known as Kasteleyn’s theorem) that allows an exact computation of the number of dimer configurations on planar graphs.

Understanding the properties of the dimer model \textit{beyond the planar setting} is significantly more challenging and of considerable interest.
In this case the model is no longer integrable and is even computationally intractable \cite{Jerrum}.
Nevertheless, important progress has recently been made in the study of the dimer problem on $\mathbb{Z}^d$ for $d \geq 3$ \cite{Chandgotia, KenyonWolfram, QuitmannTaggi, QuitmannTaggi2, T} and on weakly planar graphs \cite{GiulianiToninelli}.

In this work, we investigate the dimer model on two-dimensional not necessarily planar graphs 
and address the fundamental problem of correlation decay. 
More precisely, we study \emph{monomer--monomer correlations}, 
a central quantity in the rigorous analysis of the model \cite{FisherStephenson1963}, 
and establish a \emph{Mermin--Wagner theorem}, 
which implies decay of correlations and, consequently, 
the absence of long-range order---in sharp contrast with the behavior in dimensions $d>2$ \cite{T}.
Our main result for the dimer model is stated in Theorem \ref{thm:maintheoremdimer} below. 
In the planar setting the decay of monomer-correlations was conjectured in \cite{FisherStephenson1963} and established in \cite{Dub} using Kasteleyn's theorem and the study of Cauchy-Riemann operators.

\paragraph{The monomer double-dimer model.}
 \begin{figure}[t]
\centering
\includegraphics[scale=0.50]{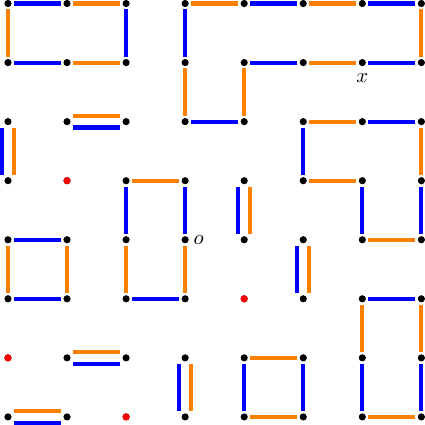}
\hspace{1cm}
\includegraphics[scale=0.50]{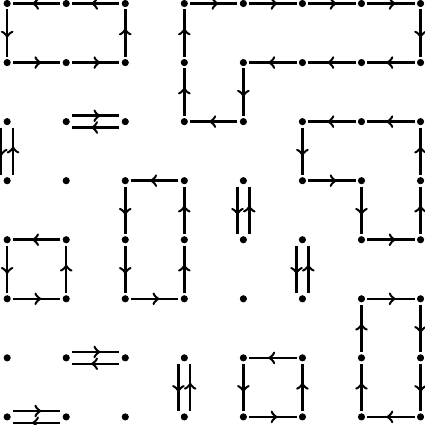}
\hspace{1cm}
\includegraphics[scale=0.41]{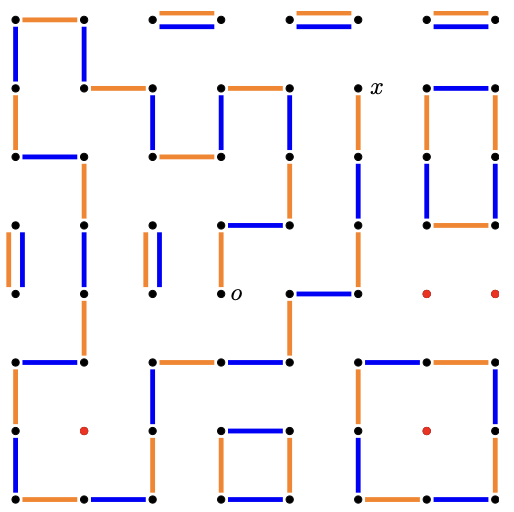}
\caption{
{\textit{Left:}} A configuration of the monomer double-dimer model in a box of $\mathbb{Z}^2$.  
{\textit{Center:}} A configuration of the spatial permutation model in a box of $\mathbb{Z}^2$, where fixed points of the permutation are represented by isolated vertices. This configuration corresponds to the one on the left.  
{\textit{Right:}} A configuration of the monomer double-dimer model in $\Omega_{o,x}$ when $x$ and $o$ have odd distance. Since $o$ and $x$ are the only vertices incident to exactly one dimer, while all other vertices are incident to either none or two dimers, the existence of a self-avoiding path with $o$ and $x$ as endpoints follows.
}
\label{fig:monomerdoubledimer}
\end{figure}
The \emph{double-dimer model} arises by superimposing two independent dimer configurations on the same graph.
The resulting random structure consists of loops and doubled edges, and provides a natural extension of the classical dimer model.
Beyond its combinatorial appeal, the model has attracted considerable attention due to its deep connections with
conformal invariance and loop ensembles.
In particular, it is conjectured that   that loops in the planar double-dimer model converge in the scaling limit to
Conformal Loop Ensembles $\mathrm{CLE}(4)$,
Kenyon proved that the scaling limit is conformal invariant  \cite{Kenyondoubledimer},
while Dub\'edat established relations to isomonodromic deformations and conformal field theory \cite{Dubedat}.
These results place the double-dimer model at the crossroads of probability, statistical mechanics, and complex analysis,
and motivate the study of its behavior beyond the planar setting \cite{GiulianiToninelli,QuitmannTaggi2}.

On planar graphs, the double-dimer model is integrable and therefore amenable to precise analysis. The picture changes dramatically once monomers are allowed. In the \emph{monomer double-dimer model}, one samples a random subset of vertices (``monomers'') that are removed together with all incident edges, and then a double-dimer configuration on the induced complementary subgraph (see Figure~\ref{fig:monomerdoubledimer}, left). A parameter---the \emph{monomer activity}---controls the monomer density. In the presence of monomers, Kasteleyn's theorem~\cite{Kasteleyn} no longer applies on planar or nonplanar graphs; the classical Kasteleyn framework only covers the special case when all monomer insertions lie on the boundary (see, e.g.,~\cite{BLQ}) and does not treat the general monomer double-dimer setting. Consequently, this model remains poorly understood and presents a challenging and compelling direction for future research.

Our main result, Theorem \ref{thm:maintheoremddmodel},  establishes the \emph{absence of long-range order} in the monomer double-dimer model
on two-dimensional graphs, for every value of the monomer activity.
More precisely, we prove decay to zero of the two-point function and the absence of macroscopic loops,
namely loops whose size is proportional to the size of the box. 
This stands in sharp contrast with the behavior in dimension $d>2$, where long-range order is known to occur and macroscopic loops are present with uniformly positive probability \cite{T,QuitmannTaggi}.
Remarkably, our result is new even for the well-studied double-dimer model,
that is, the special case of zero monomer activity,
since it applies to non-planar graphs where integrability is no longer available.

\paragraph{Spatial random permutations.}
Another object of central interest is that of \emph{spatial random permutations} on a lattice.
The configuration space is given by the set of permutations of the vertices of a finite box,
with the restriction that each vertex is mapped either to itself or to one of its nearest neighbours.
Each configuration can be represented as a collection of mutually disjoint directed self-avoiding cycles of even length,
together with vertices that are mapped to themselves (the fixed points of the permutation),
see Figure~\ref{fig:monomerdoubledimer} for an illustration.
There is a bijection between configurations of spatial random permutations on a lattice and those of the monomer double-dimer model; 
the article is written for the latter, but the results apply naturally to spatial permutations as well thanks to this correspondence 
(see Figure~\ref{fig:monomerdoubledimer}).

Random permutations are classical objects of study in probability theory and combinatorics.
The spatial version has been proposed as a toy model for the Bose gas
\cite{BetzUeltschi2008, BetzUeltschi2011, DicksonVogel2024, 
Grosskinsky2013,
Suto2009, Ueltschi2012}.
In these systems loops interact through a potential depending on their total number and of their length. 
Such interactions make the model both interesting and challenging; importantly, however,
they do not depend on the mutual spatial displacement of the loops, in contrast with the models studied here.
The presence of spatial interactions, arising from the fact that loops are self-avoiding and mutually disjoint,
makes the analysis of the model considerably more challenging and conceptually closer to that of the Bose gas.

The study of spatial random permutations in two dimensions has previously been addressed numerically in
\cite{Betz2014}, where a variant of the model was studied:
all permutations of the vertices of the box  are allowed, but the measure is weighted with Gaussian factors
penalizing jumps according to their length, thus effectively allowing only jumps of finite size,
independently of the size of the box,  just as in the  version considered here.
Based on numerical evidence, the author conjectures the existence of a Kosterlitz--Thouless transition,
and hence the absence of long-range order and of macroscopic cycles at any temperature.

In this article we present the first rigorous study of spatial random permutations on two-dimensional graphs,
we prove a Mermin--Wagner theorem for this model,  and deduce the absence of long-range order and of macroscopic cycles at any temperature. 
This is in sharp contrast with the situation in dimension $d>2$, where long-range order
has been proved to occur at sufficiently low temperature \cite{T,QuitmannTaggi2}
and where the existence of macroscopic cycles has been established in the special case of fully packed loops, that is, when fixed points of the permutation are suppressed.

\paragraph{Generalisations of our results.}
Our results extend to more general frameworks of random loops. 

A first generalisation consists in considering a model with dimers of two types, say red and blue, 
where the number of blue dimers incident to each vertex equals the number of red dimers. 
Unlike in the monomer double-dimer model, this number is not restricted to $0$ or $1$, 
but can be any non-negative integer. 
By introducing additional ``matching'' variables at each site, which match every blue dimer with a corresponding red dimer,
configurations can once again be viewed as collections of interacting loops. 
We refer to this  model as \textit{multi-occupancy double-dimer model.}
All our results extend to this more general setting, and the article is in fact written in this framework.

A further generalisation covered by our results allows not just two dimer types but \(2N\) types (labelled \(1,\ldots,2N\)) with \(N\in\mathbb{N}\). For each odd \(i\in\{1,3,\ldots,2N-1\}\) we impose the local constraint that, at every vertex, the number of incident dimers of type \(i\) equals the number of incident dimers of type \(i+1\). Introducing auxiliary pairing variables that match type \(i\) only with type \(i+1\), any configuration decomposes into a collection of loops, each alternating between colours \(i\) and \(i+1\) for some odd \(i\). 
This framework also encompasses a model closely related to the loop \(O(N)\) model (for even \(N\)), which has been studied primarily on the hexagonal lattice (see, e.g., \cite{Loop1,PeledSpinka}). The only (minor) distinction is that in the hexagonal loop \(O(N)\) model loops of length two are excluded.
Our techniques are complementary to those developed in \cite{Loop1}: 
while the latter provide lower bounds on the probability of long loops, 
our methods yield upper bounds. 
Moreover, unlike \cite{Loop1}, our approach applies only to integer values of $N$, 
but it works for any value of the inverse temperature parameter 
and extends beyond the hexagonal lattice, as long as the graph is symmetric enough 
to allow for reflection positivity (for instance, periodic boundary conditions are required). 
We address these generalisations in Section~\ref{sect:generalN}.

 \subsection{Discussion}
\paragraph{Mermin--Wagner theorem and spin representation.}
A key ingredient of our proofs is a spin representation that relates connection probabilities
and monomer correlations in the monomer double-dimer model to spin--spin correlations.
Our representation exhibits non-trivial complex dualities between the two spin components (see Remark~\ref{remark} below),
which on the one hand ensure reflection positivity of the measure and on the other hand make the model invariant under global spin rotations. 
This framework further allows one to perform a spin-wave approximation and to heuristically study
dimer correlations in terms of a Gaussian free field or a Coulomb gas.

The main obstacle to turning such heuristics into rigorous results is that,
since the measure is complex and nonpositive,
the classical strategies based on spin-wave deformations
\cite{DobrushinSchlosman,GagnebinandVelenik, MerminWagner,McBryanSpencer}
cannot be employed. Nor can one directly apply the earlier approach of~\cite{Mermin},
which also relies on positivity assumptions. 

To put this difficulty in perspective, recall that there are essentially two 
established approaches to the proof of the Mermin--Wagner theorem in the literature.  
One approach, due to Pfister \cite{Pfister} (see also the exposition in~\cite{PeledSpinka} and \cite{Velenik}), is based on entropy estimates and 
relies crucially on the positivity of the Gibbs measure (see also,  \cite{PeledMilos} for an extension to degenerately positive Gibbs measures ).
Alternatively, the original argument of Mermin and Wagner (see~\cite{MerminWagner}) is built upon the Bogoliubov functional 
inequality. This method has been extended in several directions, e.g., to quantum spin 
systems~\cite{BenassiFrohlichUeltschi2017} and to certain 
supersymmetric spin systems that admit a positive Gibbs measure representation~\cite{BHS}. 
In the latter case, the spins take values in a hyperbolic space, which can be mapped to a positive Gibbs measure, 
so that the classical proof can be adapted ((see also, the work of Kozma and Peled~\cite{KP} and Sabot~\cite{Sab} for the polynomial decay for the vertex-reinforced jump process via similar spin representations).
At the core of this line of reasoning lies a 
Bogoliubov-type inequality derived via a scalar or operator Cauchy--Schwarz 
inequality under positivity of the Gibbs (or quantum Gibbs) measure.

In contrast, our contribution is to provide a new proof of a Bogoliubov-type 
inequality,    Theorem \ref{thm:cauchyschwarz},
based on the observation that the Fourier transform of an (almost) 
convex function on the lattice is (almost) nonnegative. 
This perspective does not require positivity of the measure and therefore opens the door to extending 
Bogoliubov-type inequalities (and, consequently, the Mermin--Wagner theorem) 
to more general settings such as reflection-positive measures. 
In this way, our approach both complements and extends the existing frameworks in a novel 
direction.

Our version of the Mermin--Wagner theorem is presented in 
Theorem~\ref{thm:magnetisationbound}.
It states that, at any monomer activity, the magnetisation vanishes
when the external field is taken to zero in a suitable way
in two dimensions.  
A perturbative analysis then shows that vanishing magnetisation implies
that the Cesàro sum of two-point functions is small
(see Proposition~\ref{prop:magnexp}),
thus leading, via the Mermin--Wagner theorem, to the proof of our main results,
Theorems~\ref{thm:maintheoremdimer} and~\ref{thm:maintheoremddmodel}.

\paragraph{Overview on other related models.}
As explained above, one of the central objects of the present work is the 
monomer double-dimer model, that is, a configuration consisting of a set of monomers together with two independent dimer configurations on the complementary set of vertices.
It is therefore natural to compare it with other classical models of statistical mechanics obtained by superimposing monomers and dimers.

A first example is the \textit{monomer--dimer model}, 
defined as the superposition of a set of monomers and a single configuration of dimers on the complementary set. 
Configurations are sampled at random according to a measure depending on a parameter, 
the monomer activity, which controls the density of monomers. 
The model reduces to the pure dimer model when the monomer activity is set to zero.

This model has been studied in several works, starting from~\cite{HeilmannLieb1972}, where a Lee--Yang theorem was obtained for any value of the monomer density. 
A natural consequence of this theorem is that correlations decay exponentially fast at all strictly positive values of the monomer activity \cite{Quitmann2023}
(see also \cite{vandenBerg1999} for a probabilistic proof of this result).
Hence, in contrast to the monomer double-dimer model --- which is conjectured to exhibit a Kosterlitz--Thouless transition in two dimensions and is known to display long-range order at low monomer activity in three dimensions~\cite{QuitmannTaggi2} --- the monomer--dimer model does not exhibit any non-trivial phase transition as the monomer activity varies, in any dimension. 
An interesting open problem for this model is to characterise how the mass gap depends on the monomer activity; see~\cite[Remark 2.2]{Quitmann2023} for a discussion.

Another natural variant is the \emph{double monomer--dimer model}, obtained by superposing two independent samples of the monomer--dimer model on the same graph. Because the monomer locations in the two configurations need not coincide, the superposition is not merely a union of disjoint cycles and monomers. Instead, three types of structures appear: (i) monomers (vertices that carry a monomer in both configurations); (ii) closed cycles with edges alternating between the two samples; and (iii) open alternating paths whose two endpoints are monomers (one from each configuration). To the best of our knowledge, this model has not been analyzed in the literature.  Nevertheless, the model admits a representation within our spin formalism at positive external field, with the field playing the role of the monomer activity (in contrast to the monomer double-dimer model, where this role is played by the temperature of the spin measure). We expect that, in this regime,  no macroscopic  paths occur in any dimension.

\subsection{Definitions   and  main theorems}
In this section we present our main results for the three paradigmatic examples of our general framework.

\subsubsection{Dimer model}
Consider a finite undirected   graph $G = (V,E)$. 
A dimer configuration is a spanning subgraph of $G$ such that every vertex has degree  one. 
We let $\mathcal{D}_G$ be the set of  dimer configurations in $G$
and suppose that $G$ is such that $\mathcal{D}_G$ is non-empty. 
Given a set $A \subset V$, we let $G_A$ be the subgraph of $G$ with vertex set
$V \setminus A$ and with edge set consisting of all the edges  in $E$ which 
do not touch any vertex in $A$.
We let $\mathcal{D}_G(A)$ be the set of dimer configurations in $G_A$. 
We introduce the \textit{monomer-monomer correlation}, namely the ratio between number of dimers covers of  $G_{  \{x,y\} }$ and of $G$.
\begin{equation}\label{eq:monomerdimer}
\forall x, y \in V \quad \mathcal{C}_G(x,y) := \frac{ | \mathcal{D}_G(\{x,y\} )| }{  | \mathcal{D}_G|  }.
\end{equation}

We let $\S_{L,K}$  denote the two-dimensional slab torus, 
which we identify with the set $ 
\big ( ( - \frac{L}{2}, \frac{L}{2}] \cap \mathbb{Z} \big )  \times
\big ( ( - \frac{L}{2}, \frac{L}{2}] \cap \mathbb{Z} \big )
\times 
\big (( - \frac{K}{2}, \frac{K}{2}] \cap \mathbb{Z} \big )$. 
When $K >1$,  we then impose periodic boundary conditions in all three directions,
each vertex has then $6$ neighbours. 
In the special case $K=1$ and $L >1$,
$\S_{L,1}$   corresponds  to the planar torus, 
each vertex has then $4$ neighbours, while  
  $\mathbb{S}_{1,1}$ is  a single point. 
We use $o \in \S_{L,K}$ to denote the origin.
If $G = \mathbb{S}_{L, K}$, we use the subscript $_{L,K}$  rather than
$_{\mathbb{S}_{L, K}}$, in order to simplify the notation, and we use the same convention for the other quantities defined below.

Our  main theorem on the dimer model
provides a upper bound on the Cesaro sum of such a function 
on two-dimensional slabs.
\begin{thm}\label{thm:maintheoremdimer}
There exists $c < \infty$   such that, for any  $L, K \in 2 \mathbb{N} \cup \{1\}$ satisfying $K \leq  \sqrt{\log (L)}$, we have 
\begin{equation}\label{eq:secondtheo}
\sum\limits_{x \in \mathbb{S}_{L, K}} \frac{  \mathcal{C}_{L, K}(o,x) }{|  \mathbb{S}_{L, K} |} \leq c \sqrt{ \frac{K}{\log(L)}} 
\end{equation}
\end{thm} 
The behaviour of the monomer-monomer correlation was investigated in \cite{FisherStephenson1963}
with the aid of a general perturbation theory for Pfaffians.
In this paper it was conjectured the polynomial decay of this quantity with exponent $\frac{1}{2}$  after taking the thermodynamic limit, namely 
\begin{equation}\label{eq:conjecture}
\lim\limits_{L \rightarrow \infty} \mathcal{C}_{L, 1}(o,x) \sim |x|^{- \frac{1}{2}}
\end{equation}
in the limit of large $|x|$.
This conjecture, known as the \textit{Fisher-Stephenson conjecture}, was later solved by Dub\'edat (see \cite[Section 8]{Dub}).
See also, for example, 
\cite{FendleyMoessnerSondhi2002, HeilmannLieb1972,PriezzhevRuelle2008,
WilkinsPowell2021} for further rigorous and non-rigorous studies of monomer correlations. 

The rigorous results of these works strongly rely on Kasteleyn's theorem,
namely on the integrability property of the model on planar graphs.
Without such integrability features, it is difficult to expect sharp information on the precise decay of the correlation function.
The novelty of our result is that it is possible to obtain non-trivial information on this model in two dimensions—namely,  the decay to zero of the monomer–monomer correlation function on average, and thus the absence of long-range order in two dimensions—without relying on integrability properties.
Contrary to  (\ref{eq:secondtheo})
 the Cesaro sum of monomer-monomer correlations is known to be uniformly positive in the $d$-dimensional torus if $d \geq 3$ \cite{T}.
 
Note that  our theorem states the absence of long-range order still holds for slabs whose thickness may grow with 
$L$.

\subsubsection{Monomer double-dimer model and spatial random permutations}
\label{sect:monomerdoubledimerresults}
The presence of a positive density of monomers destroys planarity and prevents the use of Kasteleyn’s theorem.
As a consequence, the monomer double–dimer model is highly non-trivial and has remained completely unexplored.
Our main theorem provides the first rigorous result for this model in two dimensions: it establishes the  absence of macroscopic loops,   the decay of correlations for any monomer density,  and thus the absence of long-range order. 

The configuration space of the \textit{monomer double-dimer model} is the set 
$$
\Omega :=  \{ \omega= (M, d_1, d_2) \, \, : \, \, M \subset V,   (d_1, d_2) \in \mathcal{D}_{G}(M) \times \mathcal{D}_{G}(M) \}.
$$
We refer to the first element of the triplet $\omega = (M, d_1, d_2) \in \Omega$ as a set of \textit{monomers},
and to the  second and third element as set of \textit{dimers}
(respectively: blue and red dimers).
For example, a configuration of the monomer double-dimer model  with four monomers in the box of $\mathbb{Z}^2$ can be seen on the left of Figure \ref{fig:monomerdoubledimer}.
As one can see from the figure,  any  configuration $(M, d_1, d_2) \in \Omega$ can be viewed as a collection of disjoint self-avoiding loops,
with loops having length at least four and corresponding to an alternation of blue and red dimers 
or having length two and corresponding to the superposition of one blue and one red dimer on the same edge.

We define 
a probability measure on $\Omega$
which assigns to each configuration $\omega = (M, d_1, d_2) \in \Omega$ the weight
\begin{equation}
\label{eq:probddimer}
\forall \omega \in \Omega \quad
\mathbb{P}_{G, \rho}(\omega)  := 
\frac{\rho^{|M|} }{ \mathbb{Z}_{G, \rho} },
\end{equation}
where $\rho \geq 0$ is a parameter, the monomer activity, 
 and $|M|$ denotes the cardinality of the set of monomers $M$,
 and 
 $\mathbb{Z}_{G, \rho}$ is a normalising constant.
We use $\mathbb{E}_{G, \rho}$ to denote the expectation with respect to
(\ref{eq:probddimer}).

In the special case $\rho=0$ the monomers are suppressed almost surely, hence one recovers the \textit{double-dimer model}, corresponding to the superposition of two dimer covers of $G$ drawn independently and uniformly at random  (\cite{Dubedat, Kenyondoubledimer}).

In order to state our main theorem for the monomer double-dimer model we need to introduce two important observables.  The first quantity is the \textit{loop length}.
We define $L_x = {L}_x(\omega)$ the subgraph of $G$ corresponding to the loop of $\omega$ touching the vertex $x \in V$ 
in the configuration $\omega \in \Omega$
and set $L_x(\omega) = \emptyset$ if the configuration $\omega$ has a monomer at $x$ .
We denote by  $|{L}_x| = |{L}_x| (\omega)$  the length of the loop $L_x(\omega)$,  namely the number of vertices which are touched by such a loop, and set $|{L}_x|(\omega) =0$ if $\omega$ has a monomer at $x$. 

The second interesting quantity is the \textit{two-point function}, corresponding to ratio of the total weight of configurations with a walk having two prescribed vertices as end-points and the partition function.  More formally,
if  $x$ and $y$ are two vertices having odd distance, then we define  
$\Omega_{x,y}$ as the set of triplets $\omega = (M, d_1, d_2)$
such that $M \subset V \setminus   \{x,y\}$, $d_1 \in \mathcal{D}_G(  \{x,y\} \cup M )$, and
$d_2 \in \mathcal{D}_G(M)$.
If, instead,  $x$ and $y$ have even distance,
we define 
$\Omega_{x,y}$ as the set of triplets $\omega = (M, d_1, d_2)$
such that $M \subset V \setminus \{x,y\}$, $d_1 \in \mathcal{D}_G(  \{x\}\cup M )$, and
$d_2 \in \mathcal{D}_G(  \{y\} \cup M)$.
As one can see on the right of Figure \ref{fig:monomerdoubledimer}, any  configuration $\omega \in \Omega_{x,y}$ can be viewed as a collection of monomers,  mutually disjoint self-avoiding loops and a self-avoiding walk with $x$ and $y$ as end-point. 
We define the \textit{two-point function} for the monomer double-dimer model 
$$
\mathbb{G}_{G, \rho}(x,y) := 
\frac{\sum\limits_{\omega = (M, d_1, d_2) \in  \Omega_{x,y} } \rho^{ |M|  } }
{\mathbb{Z}_{G, \rho}}
$$

The first claim in our theorem states that the expected loop length  divided by the volume converges to zero in the limit of large slabs.
The second claim   states the absence of macroscopic loops.
The third claim  states the decay to zero of the two-point function and thus the absence of long-range order in the monomer double-dimer model. 
Our results hold for any value of the monomer activity.
\begin{thm}\label{thm:maintheoremddmodel}
Suppose that $\rho  \geq 0$.
There exists $c = c(\rho) \in (0, \infty)$  such that,  for any  $L, K  \in  2 \mathbb{N} \cup \{1\}$  satisfying $K \leq  \sqrt{\log (L)}$, and for any $\varepsilon >0$
\begin{align}\label{eq:firsttheo}
\frac{\mathbb{E}_{L,  K, \rho} \big (  |L_o |   \big ) }{|  \mathbb{S}_{L, K} |} & \leq c \sqrt{ \frac{K}{\log(L)}},\\ 
\label{eq:firsttheo2}
\P_{L,  K,  \rho} \big (  |L_o |  > \varepsilon  | \mathbb{S}_{ L,K  }   | \big )  & \leq c\varepsilon^{-1} \sqrt{ \frac{K}{\log(L)}}, \\
\label{eq:firsttheo3}
\sum\limits_{x \in \mathbb{S}_{ L,K  } } \frac{ \mathbb{G}_{ L, K, \rho  }(o,x) }{|\mathbb{S}_{ L,K  }| }  & \leq c \sqrt{ \frac{K}{\log(L)}}.
\end{align}
\end{thm}

Since $\mathbb{E}_{L,  K, \rho} \big (  |L_o |   \big )  = \sum_{x \in \S_{L,K}} \mathbb{P}_{L, K, \rho} (o \leftrightarrow x)$, 
where $\{x \leftrightarrow y\} \subset \Omega$ is the event that there exists a loop connecting $x$ and $y$,
 our theorem implies that the probability that two distant  vertices
are  visited by the same loop is not uniformly positive,
contrary to what happens in dimension three and higher 
and small monomer activity.

Note that, similarly to Theorem \ref{thm:maintheoremddmodel}, 
the absence of long-range order and of macroscopic loops holds true for slabs whose size can  grow even slightly faster than $O(L^2)$.

Moreover our results also hold for a generalisation of the model in which an additional multiplicative factor 
$N^{ \# \mbox{ \small loops}}$ is added to the measure,  see Section \ref{sect:generalN} below.

\paragraph{Spatial random permutations.}  
The model of spatial random permutations on a graph $G=(V,E)$ is defined as follows.  
Let $\Omega^{per}$ be the set of permutations $\pi$ of $V$ such that, for each $x \in V$,  
\[
d_G(x,\pi(x)) \in \{0,1\},
\]  
where $d_G(x,y)$ denotes the graph distance on $G$.  
We define a probability measure ${P}^{per}_{G,\rho}$ on $\Omega^{per}$ by setting, for each $\pi \in \Omega^{per}$,  
\begin{equation}
\label{eq:measurepermutation}
P_{G, \rho}^{per}(\pi) = \frac{\rho^{|F(\pi)|}}{Z^{per}_{G, \rho}},
\end{equation}  
where  
\[
F(\pi) := \{x \in V : \pi(x) = x\}
\]  
is the set of fixed points of the permutation $\pi$, $\rho \geq 0$ is a parameter, and $Z^{per}_{G, \rho}$ is the normalising constant.  

As illustrated in Figure~\ref{fig:monomerdoubledimer}, there is a one-to-one correspondence between elements of $\Omega$ (i.e., realisations of the monomer double-dimer model) and elements of $\Omega^{per}$.  
Moreover, the measure \eqref{eq:measurepermutation} assigns to each permutation the same weight as \eqref{eq:probddimer} assigns to the corresponding monomer double-dimer realisation.  
Thus,  the two models are equivalent reformulations of each other;  any result stated in this paper for the monomer double-dimer model can be restated for spatial random permutations on a lattice.

\paragraph{Open problems.}
The monomer double-dimer model is expected to undergo a
\textit{Kosterlitz--Thouless phase transition} as the monomer activity varies on two-dimensional graphs.
This means that the connection probability and the two-point function exhibit \textit{exponential decay} for large enough monomer activity (which has been proved in \cite{Betz2}, where  non-trivial estimates of the critical threshold are obtained) and \textit{polynomial decay} for small enough (but possibly positive) monomer activity.  
We may hope that the spin representation may help us to import techniques 
from spin systems in the framework of perfect matchings to make further progress. 
For example, one may hope to apply the recent methods of 
\cite{AizenmanPeled, Lis} or the technique of \cite{FrohlichSpencer}
to show that the connection probability decays \textit{not faster} than polynomially on two-dimensional graphs if the monomer activity is small enough.
Another remarkable open problem is proving the existence of macroscopic loops in dimension three and higher when the monomer activity is strictly positive but sufficiently small, a result which has been proved only at zero monomer activity \cite{QuitmannTaggi2}.
The same discussion applies to spatial random permutations, since the model is equivalent to the monomer double-dimer model.
For a closely related variant of the spatial random permutations considered here, the occurrence of a Kosterlitz--Thouless phase transition in two dimensions has been conjectured and investigated numerically in \cite{Betz2014}.

\subsection{Paper organisation and notation}
In Section  \ref{sect:spinsystem} we introduce the complex spin representation
and state our main theorem for spins.
In Section \ref{sect:randompaths} we introduce the multi occupancy double-dimer model,
discuss some special cases and make a connection between dimers and spins. 
In Section \ref{sect:RPsection} we introduce  the reflection positivity property
and discuss its consequences.
In Section \ref{sect:probabilisticestimates} we present 
some probabilistic estimates,  introduce a central quantity,  \textit{the  magnetisation},
and present a perturbative analysis in the limit of vanishing external field which allows us to lower bound the magnetisation  by the Cesaro sum of two point functions at zero external field. 
Section \ref{sect:CauchySchwarz} 
is the central part of our paper, here
we introduce our central tool,  a weak version of the Cauchy Schwarz 
for our complex measure
(a Bogoliubov-type inequality).
In
Section \ref{sect:merminwagner}
we present the proof of our main theorems.
Finally,  in  Section \ref{sect:appendix} 
we describe the generalisation to arbitrary number of colours
 and we present the  proof
of our propositions which relates  spins and dimers.

\paragraph{Notation.}
Unless otherwise specified,  the constants $c$, $c_i$, $i=1,2, \ldots$, appearing in the text will always be positive,  finite and will possibly depend on the parameter $\beta$, which is introduced below and corresponds to the inverse monomer activity in the monomer double-dimer model.
Their value may change from line to line.  
Any further dependence of such constants   from the other parameters
will be specified explicitly. 
In the whole paper the graph $G=(V,E)$ will always be an undirected finite bipartite  graph.
We write $x \sim y$ whenever $\{x, y\} \in E$.
Moreover, $d(x,y)$ will be used to denote the graph distance.
We  use the  notation $\mathbb{N}_0 = \{0, 1, \ldots\}$,
$\mathbb{N}=\{1,2,\ldots\}$, $\mathbb{R}^+_0 = \{x \in \mathbb{R}: x \geq 0\}$. 
Finally,  we use the abbreviations `LHS' for  `left-hand side' and `RHS' for `right-hand side'.

\section{The spin representation}
\label{sect:spinsystem} 
We let $G=(V,E)$ be an undirected finite bipartite graph with $o \in V$ a prescribed vertex.
We let $V^e$ (resp. $V^o$)  be the set of vertices having even  (resp. odd) graph distance from $o$. 
We let $U : \mathbb{N}_0 \mapsto \mathbb{R}^+_0$ be a \textit{weight function}, which defines our model.  In the whole paper we assume that the non negative sequence $\big (U(n)\big)_{n \in \mathbb{N}_0}$ is summable,
 that $U(n)>0$ for some $n>0$, and that $U(n) \leq 1$ for any $n \in \mathbb{N}_0$. 
We introduce a `local' spin space $\Xi =  [0, 2 \pi)^2$. 
We denote by $\Omega_{s} := \Xi^V$ the configuration space 
and define for each $\boldsymbol{s} = (s_z)_{z \in V}  \in \Omega_s$,
with $s_z = (s_z^1, s_z^2)$ the   \textit{spin at $z$},
namely the vector function
$$
S_z = (S_z^1,  S_z^2),
$$
where the function $S_z^k : \Omega_s \mapsto \mathbb{C}$ is defined for each $k \in \{1,2\}$ as 
\begin{equation}\label{eq:spinvariable}
S^k_z (\boldsymbol{s}) :=
\begin{cases}
e^{ i s_z^k  }   & \mbox{ if $k=1$ and $z$ even or $k=2$ and $z$ odd,}\\
e^{ - i s_z^k  } &   \mbox{ if $k=2$ and $z$ even or $k=1$ and $z$ odd,}
\end{cases}
\end{equation}
and interpret $s^1_z$, $s^2_z$ as the angles of the two components of the spin $S_z$.
We define for each $\boldsymbol{s} \in \Omega_s$ 
the function
\begin{equation}\label{eq:gfunction}
\gamma_z ( \boldsymbol{s}) : = \frac{1}{(2 \pi)^2}   \sum\limits_{n=0} ^{\infty} U(n)  
\big (\overline{ S_z^1( \boldsymbol{s}) S_z^2( \boldsymbol{s}) } \big )^n,
\end{equation}
and 
\begin{equation}\label{eq:alternation}
 \boldsymbol{\gamma}( \boldsymbol{s}) : = \prod_{z \in V}  \gamma_z(\boldsymbol{s}).
\end{equation}
We define the \textit{Hamiltonian} function
\begin{align}\label{eq:hamiltoninan}
H(  \boldsymbol{s} ) & : =     \sum\limits_{  \{i, j\} \in E  }  \big ( 
S_i^1 (\boldsymbol{s} )  S_j^1 (\boldsymbol{s} )  +    S_i^2 (\boldsymbol{s} )  S_j^2 (\boldsymbol{s} )  \big ) 
+  \,  h \sum\limits_{k \in V} \, 2 \, \cos(r s^1_k),
\end{align}
where $h, r \geq 0$ are some parameters.
When using the spin representation of dimer models we will later restrict to $r \in \{1,2\}$, 
but the definition of the spin measure makes sense for any $r \geq 0$.

We define the complex measures that assign to each measurable complex function 
$f : \Omega_s \to \mathbb{C}$ the weight
\begin{align}\label{eq:measure}
 \langle f \rangle_{G, \beta, h, r}   
   &:= \frac{1}{Z^{spin}_{G, \beta, h, r}} 
   \int_{\Omega_s}
   \boldsymbol{d s} \, 
   \boldsymbol{\gamma}(\boldsymbol{s}) \,
   e^{\beta H(\boldsymbol{s})} \, f(\boldsymbol{s}), \\ 
Z^{spin}_{G, \beta, h, r}  
   &:= \int_{\Omega_s}
   \boldsymbol{d s} \, 
   \boldsymbol{\gamma}(\boldsymbol{s}) \,
   e^{\beta H(\boldsymbol{s})},
\end{align}
where the integration is taken with respect to the product of Lebesgue measures on $\Omega_s$.  
Here $\beta \geq 0$ denotes the \textit{inverse temperature}, which in the monomer double-dimer model 
plays the role of the inverse monomer activity.   Using the correspondence between the spin models and graphical configurations (see Proposition \ref{prop:conversion}),
we see that  $Z^{spin}_{G, \beta, h, r}$ is real and strictly positive as long as $G$ admits a dimer configuration, namely $\mathcal{D}_G$ is non-empty. 
The parameter $h$ plays the role of the \textit{external field intensity} for our spin measure.
The parameter $r$, on the other hand, specifies the \textit{type} of external field. 
Its value determines which type of two-point function can be extracted in the case $h = 0$, 
by taking the limit $h \to 0$ in an appropriate way. 
For instance, when $r = 2$ we obtain information about the probability that a loop connects two distant vertices, 
whereas when $r = 1$ we recover the \textit{classical} two-point function, 
namely the ratio between the partition function with a walk connecting two points 
and the partition function restricted to loop configurations only.

\begin{rem}\label{remark}
When $h=0$, the spin representation is invariant under global spin rotations.  
Indeed, by our definition (\ref{eq:spinvariable}), adding the same constant to the `angles' 
$s_x^1$ and $s_x^2$ for every $x \in V$ leaves the integrand in (\ref{eq:measure}) unchanged.  
This invariance will play a crucial role in the proof of our main theorem, and it can be viewed as the analogue of the rotation invariance in $O(N)$ spin models with $N \geq 2$.  

Moreover,   due to our definition (\ref{eq:spinvariable}), (\ref{eq:alternation}) satisfies a complex \textit{chessboard property}, 
namely that odd sites are associated with functions corresponding to the complex conjugate of those associated with even sites.  
This feature is essential for establishing reflection positivity.
\end{rem}

Our ultimate goal is to obtain information on the behaviour of the two-point function when $h=0$.  
We focus on two distinct two-point functions, corresponding to the cases $r=1$ and $r=2$ in equation (\ref{eq:maintheorem}) below.  
These functions are related to two important quantities in our general double-dimer model (see Proposition \ref{prop:twopoint} for details).  

To extract information on these two-point functions, we consider the limit $h \to 0$ of our spin system for $r=1$ and $r=2$ in (\ref{eq:hamiltoninan}), respectively.
In  (\ref{eq:hamiltoninan}) the parameter $r$ controls the `type' of the external field.
We are now in position to state our main theorem on the decay of correlations in the spin system (\ref{eq:measure}).   

\begin{thm}\label{thm:maintheoremspin}
Fix $r \in \{1,2\}$,  let $\beta \geq 0$ be arbitrary. 
There exists $c = c(\beta) \in (0, \infty)$   such that, for any $L, K \in 2 \mathbb{N} \cup \{1\}$ satisfying   $K \leq  \sqrt{\log (L)}$,  we have 
\begin{equation}\label{eq:maintheorem}
0 \leq \sum\limits_{x \in \S_{L,K}}  \frac{1}{|\S_{L,K}|}  \langle  \,  \cos ( r \, s_o^1   ) \cos ( r  \, s_x^1)    \, \rangle_{L,  K, \beta, 0, 0}
\leq c \sqrt{ \frac{K}{\log(L)}} 
\end{equation}
\end{thm} 
Most of the paper is devoted to the proof of this theorem,  from which
we deduce Theorems \ref{thm:maintheoremdimer} and  \ref{thm:maintheoremddmodel}. 
Its proof is presented in Section \ref{sect:merminwagner} below.

\section{Multi-occupancy double-dimer model}
\label{sect:randomlooprepresentation}
\label{sect:randomcurrents}
\label{sect:randompaths}
We introduce a generalisation of the monomer double-dimer model in which each vertex is incident to the same number of dimers of each type ($1=$blue, $2=$red), where this number is an arbitrary non-negative integer.  
By introducing suitable pairing variables that match, at each vertex, every blue dimer with a corresponding red dimer, one can still interpret configurations in this more general setting as collections of loops.  

This framework encompasses both the monomer double-dimer model, where the number is restricted to $0$ or $1$, and the standard double-dimer model, where it is exactly $1$.

In Section \ref{sect:generalN} we further extend this construction to the case of an arbitrary (even) number of   colours.

\subsection{Definition}
In order to express correlations of our general spin measure in terms of realisations of the multi-occupancy double-dimer model 
we need to enlarge the (bipartite) graph $G= (V,E)$.

\begin{defn}[Enlarged graph]
We enlarge $G$ by adding two auxiliary vertices:  \\
- a \emph{ghost vertex} $g$, representing the effect of the external magnetic field;   \\
- a \emph{source vertex} $s$, which allows us to express correlation functions of the spin system in terms of dimers.   \\
Each of these vertices is connected by an edge to every vertex of the original graph $G$.  
We denote the enlarged graph by $G_{en} = (V_{en}, E_{en})$, where
\[
V_{en} := V \cup \{g,s\}, \quad g \neq s, \qquad
E_{en} := E \cup E_g \cup E_s,
\]
with
\[
E_g := \bigl\{ \{x,g\} : x \in V \bigr\}, 
\qquad 
E_s := \bigl\{ \{x,s\} : x \in V \bigr\}.
\]
\end{defn}
We will refer to $G=(V,E)$ as the original graph, in order to distinguish it from the enlarged graph.

\subsubsection{Configurations}
We now introduce the configurations of the generalised monomer double-dimer model, 
which correspond to the spin measure defined in Section \ref{sect:spinsystem}.  
The definition of the configuration space depends on the parameter $r$
— which is fixed to be an integer— appearing in the spin measure (\ref{eq:measure}).

\begin{defn}[Dimer cardinalities]\label{def:dimercard}
We first introduce the set of dimer cardinalities
\[
\Sigma_r := \Bigl\{ (m^1,m^2) \in \mathbb{N}_0^{E_{en}} \times \mathbb{N}_0^{E_{en}} \, : \,
\sum_{y \stackrel{en}{\sim} x} m^1_{\{x,y\}} 
= \sum_{y \stackrel{en}{\sim} x} m^2_{\{x,y\}} \ \ \text{and}\ \  
m^i_{\{x,g\}} \in r\mathbb{N}_0, \ \ \forall x \in V, \ i \in \{1,2\} \Bigr\},
\]
where the sum is over all vertices $y$ which are neighbours of $x$ in the enlarged graph (hence, including the ghost and the source vertex).
\end{defn}
For each edge $e\in E_{en}$, the value $m^i_e$ is interpreted as the 
\textit{number of dimers of colour $i \in \{1,2\}$} ($i$-dimers) occupying $e$.

By construction, each vertex $x\in V$ must be incident to the same number of blue and red dimers.  
This constraint does not apply to $s$ or $g$,  which are elements of $V^{en}$ but not of $V$. 
Moreover,  the number  of edges incident to $g$ must carry a multiple of $r$ dimers.

\begin{defn}[Match functions]
Given $m\in \Sigma$,  for each $x \in V$ we define $\mathcal{P}_x(m)$ as the set of \textit{match functions at $x$}, namely the set bijections between the blue dimers incident to $x$ and the red dimers incident to $x$.
We define the match functions as the Cartesian product $\mathcal{P}(m) = \prod_{x \in V} \mathcal{P}_x(m)$.
\end{defn}
Match functions match at each vertex of the original graph one blue dimer with a corresponding 
red dimer.  This allows us to interpret configurations a collections of open or closed paths. 
\begin{defn}[Configurations]
We denote by $\mathcal{W}$ the set of configurations
\[
\mathcal{W} := \{\, w=(m,\pi)=(m^1,m^2,\pi) : m \in \Sigma_r,\ \pi\in \mathcal{P}(m)\,\}.
\]
\end{defn}
Each path is obtained by successively pairing dimers of opposite colours.  
A path may be either \emph{closed}, lying entirely within $G$, or \emph{open}, in which case it starts and ends at $g$ or $s$.  
This is because no pairing functions are defined at $g$ and $s$, and the constraint requiring the number of blue dimers to equal the number of red dimers does not necessarily hold at these vertices.  
Consequently, the ghost and source vertices serve as sources of open paths.  
We refer to closed paths as \emph{loops}, and to open ones as \emph{walks}.

\subsubsection{Measure on double dimers}
We now introduce a measure on the space of configurations.  
As a preliminary step, we define the local time.

\begin{defn}[Local time]
\label{deflocaltime}
Given $m=(m^1,m^2)\in \Sigma_r$, the \emph{local time} at each vertex 
$x\in V$ of the original graph is
\begin{equation}\label{eq:localtime}
n_x(m) \;=\; \sum_{y \stackrel{en}{\sim} x} m^1_{\{x,y\}}
\;=\; \sum_{y \stackrel{en}{\sim} x} m^2_{\{x,y\}}.
\end{equation}
That is, $n_x(m)$ equals the total number of blue (or equivalently red) dimers incident to $x$, 
including those connected to the ghost or source vertex.  
By slight abuse of notation, for $w=(m,\pi)\in \mathcal{W}$ we also write $n_x(w)$ in place of $n_x(m)$.
\end{defn}

We now introduce a measure on $\mathcal{W}$. 
  
\begin{defn}[Measure]
\label{defmeasure}
Let $r \in \mathbb{N}$ and $h,\beta \geq 0$.  
We define the (unnormalised, non-negative) measure on $\mathcal{W}$ by assigning 
to each configuration $w=(m,\pi)\in \mathcal{W}$ the weight
\begin{equation}\label{eq:measurew}
\mu_{\beta,h,r}(w) \;:=\; \nu_{\beta,h,r}(m)
  \prod_{x\in V} \frac{U\bigl(k_x(w)\bigr)}{n_x(w)!},
\end{equation}
where
\[
k_x(w) := n_x(w) - m^2_{\{x,g\}}(w), \qquad x\in V,
\]
and
\begin{equation}\label{eq:edgeweight}
\nu_{\beta,h,r}(m) \;:=\;
\prod_{i=1}^2 
\Big ( 
   \prod_{e\in E}
    \frac{\beta^{m_e^i}}{m^i_e!}
   \prod_{x\in V}
   \frac{(\beta h)^{\,m^i_{\{x,g\}}/r}}{\bigl(m^i_{\{x,g\}}/r\bigr)!}
   \Big ) 
\end{equation}
\end{defn}

The weight of the configurations factorises into two contributions:  
  one depending only on edge cardinalities (eq.~\eqref{eq:edgeweight}), and one depending only on the local time vector.
To obtain correspondence with the spin measure when $h>0$, one must subtract the number of red dimers incident to the ghost vertex when evaluating the weight function in \eqref{eq:measurew},
this requires the introduction of the function $k_x(w)$.

When $h=0$, no edge is allowed to touch the ghost vertex, and the second product in \eqref{eq:edgeweight} equals $1$.
In this case, the measure $\mu_{\beta,h,r}$ does not depend on $r$, 
and we therefore use the simplified notation $\mu_{\beta,0,0}$.

Since $m^i_{\{x,g\}}$ is a multiple of $r$, the ratios $m^i_{\{x,g\}}/r$
are necessarily integers. 

Note that the weight $\mu_{\beta,h,r}(w)$ does not depend on the specific matching $\pi$,  but only on the number of dimers on each edge and their colour.

We use the same notation $\mu_{\beta,h,r}$ for expectations.

\begin{defn}[Expectation]
For  a function $f:\mathcal{W}\to\mathbb{R}$ we set
\[
\mu_{\beta,h,r}(f) \;:=\; \sum_{w\in\mathcal{W}} \mu_{\beta,h,r}(w)\, f(w).
\]
\end{defn}

\subsubsection{Two-point functions}
As observed before, any configuration in $\mathcal{W}$ can be viewed as a collection of paths:
either \emph{loops}, entirely contained in $G$, or \emph{walks}, which necessarily have $s$ and/or $g$ as endpoints.  
In order to represent correlation functions of the spin system introduced in Section~\ref{sect:spinsystem}, 
we need to impose further constraints involving the number of dimers on edges incident to the
source vertex.  

\begin{defn}[Source vector]
Given $w = (m, \pi) \in \mathcal{W}$, we define the \emph{source vector}
\[
\partial w = \bigl( (\partial w)^1_x, (\partial w)^2_x \bigr)_{x \in V},
\]
where $(\partial w)^i_x : \mathcal{W} \mapsto \mathbb{N}_0$ denotes the number of $i$-dimers on the edge $\{x,s\}$. 
In other words, for any configuration $w = (m^1, m^2, \pi) \in \mathcal{W}$ we have
\[
(\partial w)^i_x := m^i_{\{x,s\}}.
\]
\end{defn}

For example, take any configuration $w \in \mathcal{W}$ with no dimer touching the ghost vertex and with 
$(\partial w)^1 = \delta_x + \delta_y$ and $(\partial w)^2 = 0$, where $x$ and $y$ are vertices of $V$ at odd distance. 
By definition of the source vector, this configuration has precisely one blue dimer on the edges $\{x,s\}$ and $\{y,s\}$,
and no additional blue or red dimers incident to the source vertex.
Since on the original graph the number of blue dimers incident to each vertex equals the number of red dimers and each dimer is matched, 
it follows that the configuration consists of a walk whose last steps lie on the edges $\{x,s\}$ and $\{y,s\}$, 
together with an arbitrary collection of loops entirely contained in $G$.  
If, on the other hand, $x$ and $y$ were at even distance, no such configuration could exist, since any walk would necessarily have end-dimers of different colours.

\begin{defn}[Partition function and induced probability measure]
\label{def:probabilitymeasure}
We denote by
\[
\mathcal{W}^{\ell} := \{\, w \in \mathcal{W} : \partial w = (0,0)\,\}
\]
the set of source-free configurations.
In the special case $h=0$, any configuration in this set containing a walk receives zero weight from~\eqref{eq:measurew}. 
The \emph{partition function} of the multi-occupancy double-dimer model is defined as the total weight of all source-free configurations, namely
\begin{equation}\label{eq:partition}
{Z}^{path}_{G,\beta,h,r} := \mu_{\beta,h,r}\bigl(\mathcal{W}^{\ell}\bigr).
\end{equation}
On $\mathcal{W}^{\ell}$ we define the probability measure
\begin{equation}\label{eq:probmeasuregeneral}
P_{G,\beta,h,r}(w) := 
\frac{\mu_{\beta,h,r}(w)}{{Z}^{path}_{G,\beta,h,r}}, 
\qquad \forall\, w \in \mathcal{W}^\ell,
\end{equation}
and denote expectation with respect to $P_{G,\beta,h,r}$ by $E_{G,\beta,h,r}$.
\end{defn}

Our method allows us to obtain information on various two-point functions, but we will focus on the most important ones.
The first two-point function is the ratio of the weight of configurations with a walk and the partition function. 
The second two-point function is the probability, with respect to the probability measure defined in Definition~\ref{def:probabilitymeasure}, 
that there exists a loop connecting two sites.

\begin{defn}[Two-point functions]
For $x,y \in V$, the two-point functions are defined as
\begin{equation}\label{eq:twopointdefinition}
\mathcal{G}^{(1)}_{G,\beta}(x,y) := 
\frac{\mu_{\beta,0,0}(\,x \to y\,)}{Z^{path}_{G,\beta,0,0}},
\qquad
\mathcal{G}^{(2)}_{G,\beta}(x,y) := 
P_{G,\beta,0,0}(\,x \leftrightarrow y\,),
\end{equation}
where
\begin{itemize}
  \item $\{x \leftrightarrow y\} \subset \mathcal{W}^\ell$ is the event that there exists a loop visiting both $x$ and $y$;
  \item $\{x \to y\} \subset \mathcal{W}$ is the event that there exists a walk with one endpoint given by a dimer on $\{x,s\}$ and the other by a dimer on $\{y,s\}$, while all remaining paths are loops entirely contained in  the original graph $G$.
\end{itemize}
\end{defn}
Note that these quantities are defined under the assumption $h=0$, so walks ending at the ghost vertex are excluded.

\subsection{Special cases}\label{s.special}
Certain choices of the weight function give rise to well-known models in statistical mechanics, combinatorics, probability theory, and physics. 
We explain these connections in the most interesting case $h=0$, where the measure~\eqref{eq:measurew} takes a particularly simple form:
no dimer touches the ghost vertex almost surely, and open paths (when present) terminate only at the source vertex. 

\paragraph{Monomer double-dimer model.}
We obtain the monomer double-dimer model when $U(0) = U(1) = 1$ and $U(n) = 0$ for all $n > 1$.  
Under this choice, only configurations in which each vertex is either incident to precisely one blue and one red dimer, or to no dimer at all, are allowed.  
We let
\[
\tilde \Sigma :=  \Bigl\{
(m^1, m^2) \in \{0,1\}^E \times \{0,1\}^E : 
\sum_{y \sim x} m_{\{x,y\}}^1 = \sum_{y \sim x} m_{\{x,y\}}^2 \in \{0,1\},\ \forall x \in V
\Bigr\}.
\]
The partition function~\eqref{eq:partition} then takes the simple form 
\[
Z^{path}_{G,\beta,0,0} 
= \sum_{(m^1, m^2) \in \tilde \Sigma} \prod_{e \in E} \beta^{m^1_e + m^2_e}
= \beta^{|V|}\, \mathbb{Z}_{G,1/\beta},
\]
where $\mathbb{Z}_{G,1/\beta}$ was defined in~\eqref{eq:probddimer}.
For the first identity we used $h=0$ and the fact that, given $m \in \Sigma_r$ with each vertex $x \in V$ either incident to one blue and one red dimer or left untouched, the set $\mathcal{P}(m)$ contains exactly one element.  
For the second identity we used that the total number of monomers plus the total number of dimers always equals $|V|$.  
Hence the probability measure~\eqref{eq:probmeasuregeneral} coincides with the probability measure on $\tilde \Sigma$ assigning to each $m \in \tilde \Sigma$ the weight
\[
 \frac{\prod_{e \in E} \beta^{m_e^1 + m_e^2}}{Z^{path}_{G,\beta,0,0}}.
\]
Moreover, for any $x,y \in V$ we have
\begin{align}\label{eq:monomerdimerloop} 
\mathcal{G}^{(1)}_{G,\beta}(x,y) & = \mathbb{G}_{ G,   1/\beta }(x , y) \\
\mathcal{G}^{(2)}_{G,\beta}(x,y) & = \mathbb{P}_{G, 1/\beta}(x \leftrightarrow y).
\end{align}
where, the quantities on the right-hand side involve the monomer double-dimer model and have been defined in Section \ref{sect:monomerdoubledimerresults}.

\paragraph{Double-dimer model and dimer model.}
We obtain the double-dimer model (corresponding to the monomer double-dimer model with zero monomer activity) 
by setting $U(1)=1$ and $U(n)=0$ for all $n \neq 1$.  
In other words, each vertex is incident to precisely one blue dimer and one red dimer.  
For $\beta=1$ the partition function~\eqref{eq:partition} reduces to
\begin{equation}\label{eq:dimersquared}
Z^{path}_{G,1,0,0} = \bigl| \mathcal{D}_G(\emptyset) \bigr|^2.
\end{equation}
Moreover,
\[
\mu_{1,0,0}(x \to y) = \bigl| \mathcal{D}_G(\emptyset) \bigr| \, \bigl| \mathcal{D}_G(\{x,y\}) \bigr|
\]
(see also \cite[Figure~2]{T}).
Therefore, from~\eqref{eq:twopointdefinition} and the previous identities we obtain the following relation, which is key for the proof of Theorem~\ref{thm:maintheoremdimer}:
\begin{equation}\label{eq:monomerrelation}
\mathcal{G}^{(1)}_{G,1}(x,y) =  \mathcal{C}_G(x,y) 
\end{equation}
holding for any distinct $x,y \in V$,
where on the right-hand side we have the monomer-monomer correlation for the dimer model, which has been defined in (\ref{eq:monomerdimer}).

\paragraph{The XY model.}
We obtain the double-dimer representation of the XY model when $U(n) = 1$ for all $n \in \mathbb{N}_0$.  
In this case the random loop model~\eqref{eq:measure} is equivalent to the models introduced in \cite{Lis, LeesTaggiCMP2020} 
(to see the correspondence, one interprets the two colours of dimers as the two possible directions of edges in the multigraph considered in \cite{Lis}), 
where it is proved that such models provide a representation of the XY model.

\subsection{From spins to double dimers}
The next proposition establishes a correspondence between correlations in 
the multi-occupancy double-dimer model 
and in the spin system defined in Section~\ref{sect:spinsystem}. 
In the statement, the domain of the weight function, which has been defined in Section \ref{sect:spinsystem}, 
is extended to the negative integers, namely
$U: \mathbb{Z} \mapsto \mathbb{R}^+_0$,  under the assumption
that $U(n) = 0$ for any  $n<0$. 

\begin{prop}[From spin correlations to double dimers]\label{prop:conversion}
Let $G$ be a bipartite graph such that  $\mathcal{D}_G \neq \emptyset$. 
For $i \in \{1,2\}$, let $u^i= (u^i_x)_{x \in V} \in  \mathbb{Z}^{V}$
arbitrary vectors,
let $u^{i, \pm} = (u^{i, \pm}_x)_{x \in V} \in \mathbb{Z}^V$
be defined as  $u^{i, +}_x = u^{i}_x \vee 0$ 
and $u^{i, -}_x = u^{i}_x \wedge 0$ for each $x \in V$.
Then, for any $r \in \{1,2\}$,  $\beta,h \geq 0$,  
we have that 
$$
Z^{spin}_{G, \beta, h, r} = Z^{path}_{G, \beta, h, r} > 0
$$
and that 
\begin{multline}
   \label{eq:correlation}
 \Bigl\langle       
  \prod_{x \in V} \prod_{i=1}^{2}  (S_x^i)^{u_x^i}
  \Bigr\rangle_{G, \beta, h, r}   
  = \frac{1}{{Z}^{path}_{G, \beta, h, r}} 
\mu_{\beta, h, r}  \Bigg(
   \mathbbm{1}_{\{ \partial w  = ( u^{1,+}  - u^{2,-},\, u^{2,+} -  u^{1,-} )\}}
 \prod_{x \in V} \frac{ U\!\left( k_x(w) + u^{1,-}_x + u^{2,-}_x \right) }{ U\!\left( k_x(w) \right) } \Bigg),
 \end{multline}
\end{prop}
The previous proposition shows that averaging the function \( (S_x^i)^n \), with \( n>0 \) and \( x \in V \), enforces the presence of \( n \) \( i\)-dimers on the edge \(\{x,s\}\). 
Conversely, averaging \( {(\overline{S_x^i})}^n = (S_x^i)^{-n} \) enforces the presence of dimers of the opposite colour on \(\{x,s\}\), while simultaneously modifying the local weight assigned by the weight function at \(x\). 
The proof of Proposition~\ref{prop:conversion} is deferred to Section~\ref{sect:proofofprop}.
Note that, since the weight of the measure $\mu_{\beta, h, r}$ contain products of factors $U(k_x(w))$,
the quantity in the RHS of (\ref{eq:correlation})
is well-defined also when  the denominator $U\!\left( k_x(w) \right)$
equals zero (division by zero, then, never occurs).

The next proposition states a further correspondence between spin and dimer observables.
The proposition involves the following functions, for any $z \in V$,
\begin{equation}\label{eq:functiongu}
\psi_z(\boldsymbol{s}) = 
\frac{ U(0)\,  e^{- 2 h \cos ( r s_z^1)} }{(2 \pi)^2 \gamma_z(\boldsymbol{s})}
\qquad
\psi_{p,z}(\boldsymbol{s}) =  
    \frac{ (\overline{ S_z^1 S_z^2 ) }^p }{(2 \pi)^2  \gamma_z(\boldsymbol{s})}
\end{equation}
with $r, p \in \mathbb{N}_0$, 
whose average induces constraints on the local time.

\begin{prop}[Inducing constraints on the local time]
\label{prop:conversionlocaltime}
Let $G$ be a bipartite graph such that  $\mathcal{D}_G \neq \emptyset$. 
Then, for any $r \in \{1,2\}$, $p \in \mathbb{N}_0$, $\beta,h \geq 0$,  
the following identities hold:
 \begin{align}
    \label{eq:correlation2}
     \Big\langle      
  \prod_{z \in A} \psi_{z} 
 \Big\rangle_{G, \beta, h, r} 
     &=    P_{G, \beta, h, r} \big ( \forall z \in A,\; n_z(w) = 0  \big), \\
 \label{eq:correlation3}
 \Big\langle      
 \prod_{z \in A} \psi_{p,z} 
 \Big\rangle_{G, \beta, h, r}   
    &=    P_{G, \beta, h, r} \big ( \forall z \in A,\; k_z(w)  =  p  \big).
\end{align}
\end{prop}
The proof of the proposition is postponed to Section \ref{sect:proofofprop} as well.

We now discuss whether the average of a general complex function with respect to our complex spin measure is real and positive.  
\begin{defn}[Real-expectation and Positive-expectation functions]
\label{def:realpositiveexp}
We say that a function $f : \Omega_s \to \mathbb{C}$ is \textit{real-expectation} 
if there exists for each $x \in V$   a real-valued sequence 
$ \big (a^{x}_{n_1, n_2} \big )_{  n_1, n_2 \in \mathbb{Z} }$
 such that 
$$
f(s) = 
\sum_{\boldsymbol{n}  \in \mathbb{Z}^V \times \mathbb{Z}^V  }  
\prod_{x \in V}     a^{x}_{n_x^1, n_x^2} e^{i  ( n_x^1 s_x^1 - n_x^2 s_x^2)}
$$ 
and
\[
\int_{\Omega_s} \boldsymbol{ds}\, |f(\boldsymbol{s})| < \infty.
\]
We say that $f$ is \textit{positive-expectation} if it is real-expectation and, in addition, 
$a^{x}_{n_1, n_2} \geq 0$ for each $n_1, n_2 \in \mathbb{Z}$
and $x \in V$. 
\end{defn}
Applying 
Proposition \ref{prop:conversion}, 
and 
recalling  that $S_x^i = e^{ \pm i s_x^i}$ 
we see that, if $f$ is  real-expectation, then   $ \langle f  \rangle_{G, \beta, h, r} $ has a real  (finite) value and,  if  $f$ is positive-expectation,  it also satisfies $ \langle f  \rangle_{G, \beta, h, r} \geq 0.$

We now use the double-dimer   formulation to show that the average of certain observables for the spin system is zero when  the external field is zero.
\begin{prop}[Observables with zero average]
\label{prop:onepointiszero}
Let $G$ be a bipartite graph such that  $\mathcal{D}_G \neq \emptyset$. 
Suppose that  $h=0$ and let $\beta, r  \geq 0$ be arbitrary,
fix $n \in \mathbb{N}$ arbitrary. 
For any $x  \in V$ we have that 
\begin{equation}\label{eq:averageiszero1}
\langle  ({S^i_x} )^n \rangle_{G, \beta, h, r} = \langle  (\overline{S^i_x} )^n \rangle_{G, \beta, h, r}  =  0. 
\end{equation}
Moreover, if $x, y \in V$   have different parity, then
$$
\langle  (\overline{{S^i_x}}  S^i_y)^n \rangle_{G, \beta, h, r}   = 
\langle  ( {S^i_x}  \overline{ S^i_y})^n \rangle_{G, \beta, h, r}  = 0,
$$
while if they have the same  parity, then
$$
\langle  ({S^i_x}  { S^i_y})^n \rangle_{G, \beta, h, r}  = \langle  {( \overline{{S^i_x}  { S^i_y}})}^n \rangle_{G, \beta, h, r}   
 = 0. 
$$
\end{prop}
\begin{proof}
By Proposition~\ref{prop:conversion} we observe that 
$\langle (\overline{S^i_x})^n \rangle_{G,\beta,h,r}$ and 
$\langle (S^i_x)^n \rangle_{G,\beta,h,r}$ 
are sums over configurations 
$w\in\mathcal{W}$ such that 
$\partial w = (n\delta_x,0)$ or $\partial w=(0,n\delta_x)$. 
Since $h=0$, no path can end at the ghost vertex in such configurations. 
Any path with both end-steps on $\{x,s\}$ necessarily has dimers of different colour in its last steps. 
Hence, the corresponding sum is over configurations with zero weight, which proves the first identity. 

\medskip
We now prove the second identity. 
Since $h=0$, Proposition~\ref{prop:conversion} implies that 
$\langle (\overline{S^i_x}S^i_y)^n \rangle_{G,\beta,h,r}$ and 
$\langle (S^i_x\overline{S^i_y})^n \rangle_{G,\beta,h,r}$ 
are sums over configurations with $n$ dimers on $\{x,s\}$ and $n$ dimers on $\{y,s\}$, 
where the dimers on $\{x,s\}$ are of different type from those on $\{y,s\}$, 
and no other dimer touches the source or ghost vertex. 
This implies the existence of a path whose last steps are dimers of different type on the edges 
$\{x,s\}$ and $\{y,s\}$. 
Such a condition can be satisfied only if the distance between $x$ and $y$ is even; 
otherwise the sum is empty. 
This proves the second identity. 

\medskip
Finally, we prove the third identity. 
If $x=y$ then the first identity applies, so we assume $x\neq y$.
Again with $h=0$, Proposition~\ref{prop:conversion} implies that 
$\langle (\overline{S^i_xS^i_y})^n \rangle_{G,\beta,h,r}$ and 
$\langle ({S^i_x S^i_y})^n \rangle_{G,\beta,h,r}$ 
are sums over configurations with $n$ dimers on $\{x,s\}$ and $n$ dimers on $\{y,s\}$, 
all of the same type, and no further dimer touching the source or ghost vertex. 
This implies the existence of a path whose last steps are dimers of the same type on the edges 
$\{x,s\}$ and $\{y,s\}$. 
Such a condition can be satisfied only if the distance between $x$ and $y$ is odd; 
otherwise the sum is empty. 
This proves the third identity.

\end{proof}

The next proposition relates  the two point functions (\ref{eq:twopointdefinition}),
which are defined for the multi-occupancy double-dimer model,  to the two-point functions of the spin system appearing in  Theorem \ref{thm:maintheoremspin}.

We say that the weight function $U$ is \textit{normalized monotone} if $U(0)=1$ 
and $U(n+1) \leq U(n)$ for each $n \in \mathbb{N}_0$.
We say that the weight function is \textit{fast decaying} 
if it is normalized monotone and, in addition, 
$U(n+1) \leq K_U \frac{U(n)}{n}$ for each $n \in \mathbb{N}$, where $K_U>0$
is an absolute constant. 
For example, the weight function of the monomer double-dimer model
is fast decaying, while the weight function of the double-dimer model
is not even normalized monotone, since in this case  $U(0) = 0$. 

\begin{prop}[Two-point functions]
Let $G$ be a bipartite graph such that  $\mathcal{D}_G \neq \emptyset$. 
\label{prop:twopoint}
For any $x, y \in V$  we have that 
\begin{align}
\mbox{ if $ x \neq y $}   \quad \quad \mathcal{G}_{G, \beta}^{(1)}(x,y)  &  \leq      4 \,   \langle  \cos  (  s_x^1 )  \,   \cos   ( s_y^1 )  \rangle_{ G, \beta, 0, 0   }, 
\label{eq:twopoint1} \\
\mbox{ if $ d(x,y) \in 2 \mathbb{N}+1$}   \quad \quad  \mathcal{G}_{G, \beta}^{(2)}(x,y) & \leq  K_U \,   \langle      \cos  (  2 s_x^1 )  \,   \cos   (2 s_y^1 )  \rangle_{G,  \beta, 0, 0},
    \label{eq:twopoint2}
  \end{align}
where the first inequality holds for any weight function $U$ (satisfying the general assumptions which have been made in Section \ref{sect:spinsystem})  if $x$ and $y$ have different parity and  for any normalized monotone weight function $U$  if $x$ and $y$ have the same parity,   the second inequality  holds only if $U$ is fast decaying.
\end{prop}
\begin{proof}
We omit the subscripts from $\langle  \, \, \, \rangle_{G, \beta, 0, 0}$ for a lighter notation. 
First note that by the definition of the spin variables (\ref{eq:spinvariable}) we have 
$$
 \langle \cos (  \ell s_x^1)  \cos  ( \ell  s_y^1)  \rangle=   \frac{1}{4} \, \big\langle \big( (S_x^1)^\ell + (\overline{S_x^1})^\ell \big)
\big( (S_y^1)^\ell + (\overline{S_y^1})^\ell \big) \big\rangle.
$$
Now if $x$ and $y$ have different parity then by Proposition \ref{prop:onepointiszero} we deduce that 
$
\langle \cos (  \ell s_x^1)  \cos  ( \ell  s_y^1)  \rangle
= 
\frac{1}{4} {\langle  ( S_x^1 S_y^1 )}^{\ell} \rangle 
+
\frac{1}{4} \langle { ( \overline{S_x^1 S_y^1} )}^{\ell}  \rangle 
$.
We then 
 obtain from Proposition \ref{prop:conversion} that
\begin{multline}\label{eq:casediff}
\langle \cos (  \ell s_x^1)  \cos  ( \ell  s_y^1)  \rangle = \frac{1}{ 4 Z^{path}_{G, \beta, 0, 0} }  \Big ( \sum\limits_{ \substack{ w \in \mathcal{W} : \\ \partial w = ( \ell \delta_x + \ell \delta_y,  0)  }}  
     \nu_{\beta, 0, 0}(m) \prod_{z \in V}   \frac{ U   ( n_z(m)    )}{n_z(m)!} \\ 
     + 
     \sum\limits_{ \substack{  w \in \mathcal{W} : \\ \partial w = (0,   \ell \delta_x + \ell \delta_y)  }}  
     \nu_{\beta, 0, 0}(m) \prod_{z \in V}   \frac{ U   ( n_z(m) - \ell \delta_x(z) - \ell  \delta_y(z)   )}{n_z(m)!} \Big ).
\end{multline} 
Instead,  if $x$ and $y$ have the same parity, we obtain from Proposition
\ref{prop:onepointiszero} that 
$
\langle \cos (  \ell s_x^1)  \cos  ( \ell  s_y^1)  \rangle
= 
\frac{1}{4} \langle{ (  S_x^1 \overline{ S_y^1 })}^{\ell} \rangle 
+
\frac{1}{4} \langle { (  \overline{S_x^1 } S_y^1) }^{\ell}  \rangle 
$
and from Proposition 
 \ref{prop:conversion} that 
     \begin{multline}\label{eq:casesame}
\langle \cos (  \ell s_x^1)  \cos  ( \ell  s_y^1)  \rangle = \frac{1}{4 Z^{path}_{G, \beta, 0, 0} }  \Big ( \sum\limits_{ \substack{ (m, \pi) \in \mathcal{W} : \\ \partial m = ( \ell \delta_x ,  \ell \delta_y)  }}  
     \nu_{\beta, 0, 0}(m) \prod_{z \in V}   \frac{ U   ( n_z(m) - \ell \delta_x(z)    )}{n_z(m)!} \\ 
     + 
     \sum\limits_{ \substack{ (m, \pi) \in \mathcal{W} : \\ \partial m = ( \ell \delta_y,   \ell \delta_x) }}  
     \nu_{\beta, 0, 0}(m) \prod_{z \in V}   \frac{ U   ( n_z(m) - \ell  \delta_y(z)   )}{n_z(m)!} \Big ).
\end{multline} 
If $\ell=1$ we then realise that the first term in the brackets in  (\ref{eq:casediff})  equals $\mu_{\beta, 0, 0}(x \rightarrow y)$.
 Indeed, any configuration 
$w = (m, \pi)$
with no dimer touching the ghost vertex and 
such that  
 $ \partial m = ( \delta_x  + \delta_y, 0 )$
 (with a $1$-dimer on  $\{x,s\}$ and on $\{y,s\}$) 
necessarily has a path from  $\{x,s\}$ to $\{y,s\}$
if $x$ and $y$ have different parity. 
Similarly,  using the monotonicity property of $U$, 
we see that the first term in the brackets in   (\ref{eq:casesame}) 
is greater or equal than $\mu_{\beta, 0, 0}(x \rightarrow y)$.
 Indeed, 
any configuration  $w = (m, \pi)$ 
with no dimer touching the ghost vertex and 
such that $ \partial m = ( \delta_x,  \delta_y)$
(with a $1$-dimer on  $\{x,s\}$ and a $2$-dimer on $\{y,s\}$)
necessarily has a path  
 from $\{x,s\}$ to $\{y,s\}$ if $x$ and $y$ have the same parity.
In both cases, being the second term in the brackets non negative,  we conclude the proof of  (\ref{eq:twopoint1}).

We now prove (\ref{eq:twopoint2})
when $x$ and $y$ have different parity.
We now introduce a map $f$ which maps any configuration
in 
$ \mathcal{A} := \{w = (m, \pi) \in \mathcal{W}$ such that $\partial m= (0, 2 \delta_x + 2 \delta_y)  \} \cap \mathcal{B}$
where $\mathcal{B} = \{$no dimer touches the ghost vertex$\}$
to the configuration
$f(w) = (m^\prime, \pi^\prime) \in \{x \leftrightarrow y\}  \subset \mathcal{W}^\ell$
which is obtained by 
\begin{enumerate}[(i)]
\item switching to blue (resp.  red)  the colour of each red (resp. blue) dimer belonging to one of the two open paths from $\{x,s\}$ to $\{y,s\}$, 
\item removing the two dimers on $\{x,s\}$ and the two dimers on 
 $\{y,s\}$,
\item  for each $z \in \{x,y\}$   pairing together at $z$ the
two dimers which were paired to those which have been removed. 
\end{enumerate}
Thus,   $f(\mathcal{A}) = \{x \leftrightarrow y\}   \subset \mathcal{W}^\ell$
and $f^{-1}(\{x \leftrightarrow y\}  ) =  \mathcal{A}$.
We observe that for any configuration $w =(m, \pi) \in \{ x \leftrightarrow y \} \subset \mathcal{W}^\ell$ and $w^\prime = (m^\prime, \pi^\prime) \in f^{-1}(w)$ 
we have that, for each $z \in V$, 
\begin{equation}\label{eq:multi2}
|f^{-1}(w)| \geq 1,  \quad \quad n_z(m^\prime) = n_z(m) + \delta_x(z) + \delta_y(z)
\quad \quad  \nu_{\beta, h, r}(m) =  \nu_{\beta, h, r}(m^\prime).
\end{equation}

Using (\ref{eq:multi2}) for the first inequality and the property that $U$ is fast decaying for the second inequality we obtain that 
\begin{align*}
\mathcal{G}^{(2)}_{G, \beta}(x,y)   = &   \sum\limits_{ w  = (m, \pi) \in \{x \leftrightarrow y\}  }   
     \nu_{\beta, 0, 0}(m) \prod_{z \in V }   \frac{ U   ( n_z(m)   )}{n_z(m)!} \\
        = &  \sum\limits_{ w  = (m, \pi) \in \{x \leftrightarrow y\}  }   
   \frac{1}{ | f^{-1}(w)|  }   \sum\limits_{ w^\prime =(m^\prime, \pi^\prime)  \in f^{-1}(w)   }   
     \nu_{\beta, 0, 0}(m) \prod_{z \in V }   \frac{ U \big  ( n_z(m)      \big )}{n_z(m)!}  \\
      \leq  &  \sum\limits_{ w  = (m, \pi) \in \{x \leftrightarrow y\}  }   
   \sum\limits_{ w^\prime =(m^\prime, \pi^\prime)  \in f^{-1}(w)   }   
     \nu_{\beta, 0, 0}(m^\prime) \prod_{z \in V}   \frac{ U \big  ( n_z(m^\prime)     -  \delta_x(z) -   \delta_y(z)  \big )}{ \big (n_z(m^\prime) - \delta_x(z) - \delta_y(z)\big )!}    \\
     \leq  & K_U \, \, \sum\limits_{ w^\prime  = (m^\prime, \pi^\prime) \in \mathcal{A}   }   
     \nu_{\beta, 0, 0}(m^\prime) \prod_{z \in V}   \frac{ U \big  ( n_z(m^\prime)     - 2 \delta_x(z) - 2  \delta_y(z)  \big )}{n_z(m^\prime)!},
\end{align*}
where $K_U > 0$ denotes the decay constant associated with the fast--decaying weight function $U$. 
Observing that the last term coincides with the second term inside the brackets in~\eqref{eq:casediff} when $\ell=2$, 
we thereby complete the proof of~\eqref{eq:twopoint2}.
\end{proof}

\section{Reflection positivity and its consequences}
\label{sect:RPsection}
In this section we show that the complex measure $\langle \, \, \,  \rangle_{L,K, \beta, h, r}$ which has been introduced in Section \ref{sect:spinsystem} is reflection positive and  use this property to
derive convexity and monotonicity properties for the two-point function,
following \cite{LeesTaggiCMP2020},
and to derive a Chessboard estimate,
which is a well-known consequence.
The reflection positivity property of the 
complex measure   reflects an equivalent reflection  positivity property of the double-dimer model \cite{T}.
It  relies on the complex duality relation between vertices having opposite parity
in (\ref{eq:measure}) and its proof is classical \cite[Chapter 10]{Velenik}.

In the whole section we fix
arbitrary values $\beta, h, r \geq 0$, $L, K \in 2 \mathbb{N} \cup \{1\}$,
and sometimes omit the subscripts 
of $\langle \, \, \,  \rangle_{L, K, \beta, h, r}$ for a lighter notation.

\subsection{Reflection positivity}
\label{sect:RP}
Recall that $\S_{L,K}$ is the slab torus.
We let $\S_{L,K}^o \subset \S_{L,K}$ be the set of odd sites, and 
$\S_{L,K}^e \subset \S_{L,K}$ be the set of even sites.

We let $\mathcal S$ be a plane  which is orthogonal to the Cartesian vector $e_i$ for some  $i \in \{1,2,3\}$ and which intersects the midpoint of some edges.

\begin{defn}[Reflection of sites, configurations and functions]

We let $\Theta : \S_{L,K} \mapsto \S_{L,K}$ be a reflection operator
 with respect to the plane  $\mathcal S$.
 
By a slight abuse of notation we keep using  $\Theta : \Omega_s  \mapsto \Omega_s $
to denote  a reflection operator  that associates to each configuration 
$\boldsymbol{s} = (  s_x )_{x \in V}$ the reflected spin configuration 
$\Theta ( \boldsymbol{s} ) = ( s_{ \Theta(x)}   )_{x \in \S_{L,K}}$.

Given a function $f : \Omega_s \mapsto \mathbb{C}$, 
we define the reflected function $\Theta f$ as,
$$
\forall \boldsymbol{s} \in \Omega_s \quad \Theta f ( \boldsymbol{s}) : = {f( \Theta(\boldsymbol{s}) )}.
$$ 
\end{defn}

\begin{defn}[Function domain]
We say that the function  $f : \Omega_s \mapsto \mathbb{C}$ has domain $A \subset \S_{L,K}$ if 
 for each pair of configurations
$ \boldsymbol{s},  \boldsymbol{s}^\prime \in \Omega_s$
such that $s_x = s_x^\prime$ for each $x \in A$, 
$f( \boldsymbol{s} ) = f( \boldsymbol{s}^\prime )$.
\end{defn}

For example, the spin function $S_x^k$,  $k=1,2$,  defined in (\ref{eq:spinvariable}), 
has domain $\{x\}$.

\begin{defn}[Torus halves and $\mathcal{A}^{\pm}$-functions ]
We denote by  $\S_{L,K}^\pm$  two disjoint, connected,  torus halves,
$\S_{L,K}^\pm$,
 such that $\Theta(  \S_{L,K}^\pm )
= \S_{L,K}^\mp$.

We let $\mathcal{A}^\pm$ be the set of functions $f : \Omega_s \mapsto \mathbb{C}$ which have domain  $\S_{L,K}^\pm$.
\end{defn}

Our spin variables satisfy the following identity, which is crucial for establishing reflection positivity of the spin measure: for each $x \in \mathbb{S}_{L,K}$,
\begin{equation}\label{eq:complexalternation}
\Theta(S_x) = \overline{S_{\Theta(x)}} = (\overline{S^1_{\Theta(x)}}, \overline{S^2_{\Theta(x)}}).
\end{equation}
The validity of~\eqref{eq:complexalternation} relies on the fact that the spin variables were defined in (\ref{eq:spinvariable})
with a complex chessboard  property, 
see Remark~\ref{remark}.
This condition further implies that the functions $\gamma_z$,
appearing in the definition of the spin measure and being defined in (\ref{eq:gfunction}),
satisfy for each $x \in \mathbb{S}_{L,K}$,
\begin{equation}\label{eq:complealternationgamma}
\Theta(\gamma_x) = \overline{  \gamma_{\Theta(x)}   }.
\end{equation}
Given these  key properties of the spin representation,  the proof strategy  is classical and can be found in several papers about reflection positivity, see for example \cite[Chapter 10]{Velenik}.
\begin{prop}[Reflection positivity]\label{prop:reflectionpos}
For any  $f, g \in \mathcal{A}^+$ we have 
\begin{enumerate}[(i)]
\item  $\langle  f  \,  \,  \overline{ \Theta g }    \rangle_{L, K, \beta, h,r   }  = { \langle   \overline{ g }  \,   \, \Theta f \rangle  }_{L, K, \beta, h, r   },$
\item $ \langle f \,  \overline {  \Theta f} \rangle_{L, K, \beta, h, r   }   \geq 0$.
\end{enumerate}
This in turn implies that 
\begin{equation}\label{eq:RPstatement}
{ \Big ( { {Re} \,   \langle f  \, \overline{  \Theta g }  \,  \rangle  }_{L, K, \beta, h, r   }  \Big)}^2  \leq  \langle \, f  \,  \overline{ \Theta f}  \rangle_{L, K, \beta, h, r   }    \,  \langle g  \,  \overline{ \Theta g} \rangle_{L, K, \beta, h, r   } .
\end{equation}
\end{prop}
\begin{proof}
Let us define for any measurable function $f: \Omega_s \mapsto \mathbb{C}$
the operator 
$$
\langle f \rangle_0  := \int_{\Omega_s}  f(\boldsymbol{s})\, d \boldsymbol{s}.
$$
Being $\langle \, \,  \rangle_0 $ a product measure it is easy to see that,
the measure is reflection positive,  namely 
for $f, g \in \mathcal{A}^+$ 
\begin{equation}\label{eq:positivity}
\langle f  \overline{ \Theta(f)}  \rangle_0  \geq 0,
\end{equation}
and  
\begin{equation}\label{eq:symmetry}
\langle f  \, \,  \overline{ \Theta(g)}  \rangle_0   = 
\langle \overline{ g  }  \, \,  { \Theta(f)}  \rangle_0,
\end{equation}
(see   \cite[Chapter 10]{Velenik}).
We now define for any measurable function $f: \Omega_s \mapsto \mathbb{C}$
the operator 
$$
\langle f \rangle_1  := \int_{\Omega_s}  f(\boldsymbol{s}   )\,   \boldsymbol{\gamma}(\boldsymbol{s})  d \boldsymbol{s},
$$
where $\boldsymbol{\gamma}(\boldsymbol{s}) = \prod_{z \in \S_{L,K}} \gamma_z(\boldsymbol{s})$
and $\gamma_z(\boldsymbol{s})$ has been defined in (\ref{eq:gfunction}).
Define also 
$$
\boldsymbol{\gamma}^{\pm}(\boldsymbol{s}) = \prod_{z \in \S^{\pm}_{L,K}} \gamma_z(\boldsymbol{s}).
$$

We see that 
$\boldsymbol{\gamma}^{\pm} \in \mathcal{A}^\pm$ and that,
since (\ref{eq:complealternationgamma}) holds, 
$
\Theta(  \boldsymbol{\gamma}^{\pm} )  = \overline{ \boldsymbol{\gamma}^{\mp}}. 
$
Hence, since $\langle \,  \,  \rangle_0$ is reflection positive, we deduce that also 
$\langle  \, \, \rangle_1$  is reflection positive,
namely for any $f, g \in \mathcal{A}^+$
\begin{equation}\label{eq:positivity1}
\langle f  \overline{ \Theta(f)}  \rangle_1 =  
\langle  (  f  \boldsymbol{\gamma}^{+} )    \, \,  \overline{ \Theta \big (f  \boldsymbol{\gamma}^{+}  \big ) }   \rangle_0   \geq 0,
\end{equation}
and 
\begin{equation}\label{eq:symmetry1}
\langle f  \, \,  \overline{ \Theta(g)}  \rangle_1   = 
\langle  (  f  \boldsymbol{\gamma}^{+} )    \, \,  \overline{ { \Theta \big (g  \boldsymbol{\gamma}^{+}}} \big )   \rangle_0 =   
\langle  { \overline{  \big (g  \boldsymbol{\gamma}^{+}}} \big )   \, \, 
\Theta \big ( {f  \boldsymbol{\gamma}^{+}  }   \big )   \, \,    \rangle_0 = 
\langle \overline{ g  }  \, \,  { \Theta(f)}  \rangle_1.
\end{equation}

We now prove that  the operator 
$\langle   \, \,   \rangle_{L, K, \beta, h,r   }$
is reflection positive. 
 For this,   let $E^{\pm}$ be the subset of
edges of $\S_{L,K}$ which are entirely contained in  $\S_{L,K}^{\pm}$
and $E_R$ be the set of edges
intersecting both  $\S_{L,K}^{-}$ and  $\S_{L,K}^{+}$.
We observe that we can write
for any $f, g \in \mathcal{A}^+$ 
\begin{equation}\label{eq:RP1}
\langle  f  \,  \,  \overline{ \Theta g }    \rangle_{L, K, \beta, h,r   } 
= \langle  f  \, \, \overline{ \Theta g } \, \, e^{  H_+      }   e^{  H_-} e^{  \beta  \sum_{\small \{x,y\} \in E_R}   S_x^1  \, \,  
\overline{ \Theta ( S_x^1  )  }}
e^{ \beta  \sum_{  \small \{x,y\} \in E_R}    S_x^2  \, \,   \overline{\Theta ( S_x^2 )  }}
\rangle_1
\end{equation}
 where  
 $$
 H_\pm(\boldsymbol{s})  : = 
\beta  \sum\limits_{  \{i, j\} \in E^{\pm}  }  S_i^1 S_j^1  +   \beta   \sum\limits_{  \{i, j\} \in E^{\pm} }   S_i^2 S_j^2 
+  \,  \beta \, h \sum\limits_{k \in \S_{L,K}^\pm } \, 2 \, \cos(r s^1_k)
 $$
 and we used the convention that if $\{x,y\} \in E_R$, then $x \in \S_{L,K}^+$
 and $y \in \S_{L,K}^-$.
 Noting that $e^{ H_\pm}  =  \overline{\Theta e^{
  H_\mp}}$,  expanding the last two exponentials in (\ref{eq:RP1}) 
  we obtain
  \begin{multline}
  \langle  f  \,  \,  \overline{ \Theta g }    \rangle_{L, K, \beta, h,r   }
  \\ =  \sum\limits_{m^1, m^2  \in  {\mathbb{N}_0}^{E_R} } 
  \Big  \langle  f    \, \, e^{  H_+      }  
 \,   \overline{  \Theta \Big ( 
   g \,   e^{  H_+}  \,  \Big ) }
   \prod_{ \{x,y\} \in E_R }  
 \frac{  \big (    \beta  \, S_x^1 \,   \, \,   \overline{   \Theta (  S_x^1 } )  \Big )^{m_{\{x,y\}}^1}}{m_{\{x,y\}}^1!}
 \frac{\big ( \beta   S_x^2   \, \,  \overline{  \Theta (  S_x^2 } )  \Big )^{m_{\{x,y\}}^2}}{m^2_{\{x,y\}}!}
\Big \rangle_1.
  \end{multline}
We deduce from~(\ref{eq:positivity1}) that, for general measurable functions 
$f, g \in \mathcal{A}^+$, each term in the previous sum can be written as 
$
\langle h \, \overline{\Theta(h')} \rangle_1
$
for some $h, h' \in \mathcal{A}^+$.
Using the symmetry property~(\ref{eq:symmetry1}), we have
$
\langle h \, \overline{\Theta(h')} \rangle_1
=
\langle \overline{h'} \, \Theta(h) \rangle_1.
$
Summing over all terms then yields the first claim of the proposition.
Moreover, if $f = g$, we deduce that each term in the previous sum can be written 
as  $   \langle h \, \overline{\Theta(h)} \rangle_1  $
for some $h  \in \mathcal{A}^+$, and is then non-negative by (\ref{eq:positivity1}).
This gives the second claim and concludes the proof. 
\end{proof}

\subsection{Convexity and monotonicity of two-point functions}

Reflection positivity is useful, as it allows one to derive monotonicity properties of two-point functions,
such us the functions in the RHS of (\ref{eq:twopoint1}) and (\ref{eq:twopoint2}). 
Such functions can be expressed as averages of products of two components of a reflection invariant vector function, defined below.  

\begin{defn}[Reflection invariant vector function]
A family of functions 
\[ F = (F_x)_{x \in \S_{L,K}}, \qquad F_x : \Xi \to \mathbb{R}, \]
is called \textit{reflection invariant} if for all $x,y,z \in \S_{L,K}$ the following hold:
\begin{enumerate}[(i)]
\item $\Theta F_x = F_{\Theta x}$ for any reflection $\Theta$ through edges,
\item $\langle F_x F_y \rangle = \langle F_{x+z} F_{y+z} \rangle \geq 0$,  
where addition is understood with respect to the torus metric.
\end{enumerate}
\end{defn}
The next lemma will allow us to derive convexity properties for  two-point functions.
\begin{lem} 
\label{lemma:key1}
Consider any reflection plane $\mathcal{S}$,   let $\Theta$ be the corresponding reflection operator and  $\S_{L,K}^{\pm} \subset \S_{L,K}$  the corresponding torus halves
as in Section \ref{sect:RP}.
Let $(F_x)_{x \in \S_{L,K}}$ 
be a reflection invariant vector of functions such that
$F_o$ has domain $D \subset \S_{L,K}$, with $o \in D$,
and  define $\ell := \sup \{  d(o,x) : x \in D  \}$.
Let $Q\subset \S_{L,K}$,  $Q^\pm := (Q\cap \S_{L,K}^\pm) \cup \Theta (Q\cap \S_{L,K}^\pm) $,   assume that $|Q^+| = |Q^-|$ and  that $d_H(Q^\pm,  \mathcal S) >\ell$,
where $d_H$ is the Hausdorff distance with respect to the $L^2$ norm.  
Then,
\begin{align}\label{eq:keylemma1}
2\sum\limits_{ \substack{ x,y\in Q: \\  x\neq y}} \langle F_x  F_y \rangle_{L, K, \beta, h, r} \leq
\sum\limits_{ \substack{ x,y\in Q^+: \\  x\neq y}}  \langle F_x F_y \rangle_{L, K, \beta, h, r}
+
\sum\limits_{ \substack{ x,y\in Q^-: \\  x\neq y}}  \langle F_x F_y \rangle_{L, K, \beta, h, r}.
\end{align}
\end{lem}
\begin{proof}
Let $\varphi >0$ be an arbitrary real number.   
From Proposition \ref{prop:reflectionpos} we deduce that, 
$$
 \langle      \prod_{  x \in Q    }   (  1 + \varphi F_x  ) \rangle^2 \leq 
 \langle      \prod_{  x \in Q^+   }   (  1 + \varphi F_x  ) \rangle 
  \langle      \prod_{  x \in Q^-     }   (  1 + \varphi F_x  ) \rangle.
$$
  Since the inequality holds for each $\varphi$, we perform an expansion in the limit as $\varphi \rightarrow 0$ and we obtain
    (\ref{eq:keylemma1}) from a comparison of the terms  $O(\varphi^2)$.
\end{proof}

The next proposition derives a convexity property along the Cartesian axis for two point functions.
Our result holds for two-point functions which can be expressed as the expectation of the product of two functions belonging to a reflection invariant vector.

\begin{prop}[Convexity]
\label{prop:convexity}
Let $F = (F_x)_{x \in \S_{L,K}}$ be a reflection invariant vector of functions
satisfying the assumptions in Lemma \ref{lemma:key1}.
Define for any $x \in \S_{L,K}$
\begin{equation}\label{eq:expressiontwopoint}
\mathcal{G}(x) :=  \langle F_o F_x  \rangle_{L, K, \beta, h, r}
\end{equation}
Let $x_n$ correspond to the vertex $(n, 0, 0)$ or to the vertex $(0,n, 0)$,
let $y_k$ correspond to the vertex $(0,0,k)$. 
For odd vertices $n, k$ such that  $ \min\{2\ell+1, 3\} \leq n \leq  L-  \min\{2\ell+1, 3\} $
and $ \min\{2\ell+1, 3\} \leq k \leq  K-  \min\{2\ell+1, 3\} $
 we have that, 
\begin{equation}\label{eq:convexity}
\begin{split}
2  \,  \mathcal{G}  (  x_{n}  )  &  \leq 
\mathcal{G}  (   x_{n+2}   ) + \mathcal{G}  (  x_{n-2}  ),  \\
2  \,  \mathcal{G}  (  y_{k}  )  &  \leq 
\mathcal{G}  (   y_{k+2}   ) + \mathcal{G}  (  y_{k-2}  ).
\end{split} 
\end{equation}
\end{prop}
\begin{proof}
The proposition follows  from am immediate application of Lemma \ref{lemma:key1} with
$\mathcal{S} = \{ (x_1, x_2, x_3) \in \S_{L,K} \, : \, x_1 = \frac{n}{2} - 1  \}  $
and
 $Q= \{ (0,0, 0),  (n,0, 0)\}$, or 
 $\mathcal{S} = \{ (x_1, x_2, x_3)  \in \S_{L,K} \, : \, x_2 =  \frac{n}{2} - 1  \}  $
 and
$Q= \{(0, 0, 0 ), (0,n, 0 )\}$
or 
$\mathcal{S} = \{ (x_1, x_2, x_3) \in \S_{L,K} \, : \, x_3 = \frac{k}{2} - 1  \}  $
and 
$Q= \{(0, 0, 0 ), (0,0, k)\}$.
\end{proof}
The next property is a consequence of convexity,
namely it is a monotonicity property for the two-point function.
\begin{prop}[Monotonicity]
\label{prop:monotonicity}
Under the same assumptions as in Proposition \ref{prop:convexity} we have that,
for any odd integer $n \in [0, \frac{L}{2}-1)$ and $k \in  [0, \frac{K}{2}-1)$,
$$
\mathcal{G}(x_{n}) \geq \mathcal{G}(x_{n+2}) \quad \quad \mathcal{G}(x_{k}) \geq \mathcal{G}(x_{k+2}).
$$
\end{prop}
For the proof of the proposition (given the analogous of Proposition \ref{prop:reflectionpos})
we refer to \cite{LeesTaggiCMP2020}.

\subsection{Chessboard estimates}

A classical consequence of reflection positivity is the chessboard estimate.
Let $f$ be a function having domain $\{o\}$,
let $e_1, e_2 \ldots e_k$ a sequence of edges of the graph $\S_{L,K}$  forming a self-avoiding path from $o$ to $t \in \S_{L,K}$,
let $\Theta_1, \ldots \Theta_k$ be a sequence of reflection
through edges such that  $\Theta_i$ is a reflection with respect to the plane
orthogonal to $e_i$ and intersecting the midpoint of such an edge
for each $i \in \{1, \ldots, k\}$.
Define, 
$$
f^{t} : = \begin{cases}
\overline{\Theta_k  \circ \Theta_{k-1} \, \ldots \, \circ \Theta_1 \, ( f )} &  \mbox{ if $k$ is even,} \\
\Theta_k  \circ \Theta_{k-1} \, \ldots \, \circ \Theta_1 \, ( f ) &  \mbox{ if $k$ is odd.}
\end{cases}
$$
Observe that the function $f^{t}$ does not depend on the chosen path
and has domain $t$. 
\begin{prop}[Chessboard estimate]\label{prop:chessboardabstract}
Let $f = (f_t)_{t \in \S_{L,K}}$ be a sequence of  real expectation   functions
having domain  $\{o\}$.  
We have that,
$$
\langle   \prod_{t \in \S_{L,K}} f^{t}_t   \rangle_{L, K, \beta, h, r}  \, \leq \, 
\prod_{t \in \S_{L,K}} \, \, 
{ \big  \langle 
\,    \prod_{s \in \S_{L,K}} f_t^{s}   \,  \big  \rangle_{L, K, \beta, h, r} }^{\frac{1}{|\S_{L,K}|}}
$$
\end{prop}
Given Proposition \ref{prop:reflectionpos}, the proof of Proposition \ref{prop:chessboardabstract} is classical and is based on performing a sequence of reflections, we refer for example to \cite[Theorem 10.11]{Velenik}
for its proof.

\section{Probabilistic estimates}
\label{sect:probabilisticestimates}
In this section we 
use reflection positivity and probabilistic arguments
to provide basic estimates for some quantities of interest which are defined in the framework of the spin system.  
Most of our proofs however are formulated in the framework of the multi-occupancy double-dimer model.

In the whole section we use $G=(V,E)$ to denote the graph corresponding to the slab torus $\S_{L,K}$,
we let $d$ be the degree of the origin
(which might be $4$ or $6$),
and  we  omit the subscripts 
from $\langle \, \, \,  \rangle_{L, K, \beta, h, r}$ when unnecessary for a lighter notation.
We assume that $L$ is even,  thus ensuring that $\mathcal{D}_G \neq \emptyset$.

The following technical lemma is a consequence of the chessboard estimate. 

\begin{lem}[Chessboard upper bound]\label{lem:chessboardbound}
Fix $\beta >0$, 
Let $f_1, f_2 : [0, 2 \pi)\mapsto  \mathbb{R}$ be two functions which admit a Fourier series with a bounded numbers of terms
and have all real-valued coefficients, namely there exists $R \in \mathbb{N}_0$ and 
 some   coefficients $a^j_{-R},  \ldots a^j_R \in \mathbb{R}$
 such that 
for each  $s \in [0, 2 \pi)$ and $j \in \{1,2\}$
$$
f_j(s)= \sum_{n=-R}^{R} a^j_n e^{i \,  n  \, s}.
$$
Define 
$$
M :=  \max \big \{  \max \{ |a^j_n| \, :  -R \leq n \leq R,  j \in \{1,2\}  \}, \,  1 \big \}.
$$
Then,
 there exists $c = c ( \beta) \in (0, \infty)$ such that,  
for any $h \in [0, 1]$,  $i, j \in \{1,2\}$,  
 $L, K \in 2 \mathbb{N} \cup \{1\}$,
and disjoint sets $A, B \subset \S_{L,K}$ 
 \begin{equation}\label{eq:chessboundcos}
 \langle  \prod_{x \in A} f_1(  s_x^i )   \prod_{x \in B} f_2(  s_x^j)  \rangle  \leq ( c \, M \,  (R+1) ) ^{|A| + |B|  }.
 \end{equation}
\end{lem}
\begin{proof}
From Proposition \ref{prop:chessboardabstract} we deduce that
\begin{equation}\label{eq:startingchessboard}
 \langle  \prod_{x \in A} f_1(   s_x^i)
 \prod_{x \in B} f_2(   s_x^j)   \rangle  \leq 
 \Big (  \big  \langle  \prod_{x \in V} f_1(  s_x^i) \big  \rangle \Big )^{\frac{|A|}{|V|}} 
 \Big (  \big  \langle  \prod_{x \in V} f_2(  s_x^j) \big  \rangle \Big )^{\frac{|B|}{|V|}} 
\end{equation}
Fix $i, j \in \{1,2\}$.
We have 
\begin{align*}
 &  \langle  \prod_{x \in V} f_j( s^i_x )        \rangle   =
\sum\limits_{ \boldsymbol{n} \in \{-R,  \ldots,R \}^{V}    }
\langle     \prod_{x \in V}  a^j_{n_x} e^{  i \,  \, n_x  \, s^i_x  }          \rangle  \\
& \leq  {(2 R+1)}^{|V|} \, \,  {M}^{|V|}
\max  \Big \{            \frac{1}{Z^{path}_{G, \beta, h, r} } \sum\limits_{ \substack{ m \in \Sigma_r : \\ \partial m = (\boldsymbol{n}^1, \, \,   \boldsymbol{n}^2) } }  
   { \nu }_{\beta, h, r}(m)  
       \, \, : \, \, \boldsymbol{n}  = ( \boldsymbol{n^1},  \boldsymbol{n^2} ) \in  \{-R, \ldots, R  \}^{V \times V}   \Big     \}
       \\
   &     \leq     c^{|V|}
   {(2 R +1)}^{|V|} \, \,  {M}^{|V|}
 e^{ 2  \beta |E| + 2  \beta h   |V|    }.
\end{align*} 
For the first inequality we used Proposition \ref{prop:conversion}
to express the average with respect to the spin
measure as an average with respect to the  measure of the multi-occupancy double dimer model
(recall that, by Proposition \ref{prop:conversion},  taking the average with respect to
$(S_x^i)^n$, which equals $e^{i n s_x^i}$ or $e^{-i n s_x^i}$,
induces the constraint that there are $n$ dimers on $\{x,s\}$ of type $1$ or $2$).
We then used the assumption that  $U(n) \leq 1$ for each $n \in \mathbb{N}_0$,
we observed that the number of possible matchings at $x$ is $n_x(w)!$,
thus cancelling the  factor $n_x(w)!$ in the denominator of (\ref{eq:measurew}),
and we took the maximum over all possible ways to choose the number of
dimers on each edge touching the source vertex. 
For the second inequality we used the  definition
of the weight $\nu_{\beta, h, r}$ (see Definition \ref{defmeasure})
and summed  over the number of dimers  independently for  each edge $e \in E \cup E_g$ and each colour. 
We also used the fact that  
\begin{equation}\label{eq:lowerboundZ}
Z^{path}_{G, \beta, h, r}  \geq c^{|V|},
\end{equation}
for some $c = c (\beta) \in (0, \infty)$.
Indeed, since  $G$ admits at least one dimer configuration,
in order to lower bound $Z^{path}_{G, \beta, h, r}$ one can fix a dimer configuration
and use the lower bound $Z^{path}_{G, \beta, h, r} \geq \mu_{\beta, h, r}(w)$,
where $w \in \mathcal{W}^{\ell}$ is the configuration  which has exactly $n$ loops
of length two on each edge hosting a dimer, where $n$ is a fixed integer such that  $U(n)>0$. This gives the desired lower bound on $Z^{path}_{G, \beta, h, r}$.
\end{proof}

\begin{lem}\label{lem:mestimates}
For any $\beta >0$,  $\epsilon >0$,
there exists a constant $c = c(\beta)> 0$  and $M = M(\beta, \epsilon) < \infty$ such that for any $h \in [0, 1]$,  $r \in \{1,2\} $,   $L, K \in 2 \mathbb{N}  \cup \{1\}$,   $z,o \in V$ neighbours, 
\begin{align}
 \label{eq:mestimates1} 
P_{L, K, \beta, h, r}  \big ( m_{\{o,z\}}^i  > M   \big ) & \leq \epsilon, \\
 \label{eq:mestimates2}
 P_{L, K, \beta, h, r}  \big ( n_o> 0    \big ) & \geq  c.
\end{align} 
\end{lem}
\begin{proof}
The following inequalities and identities hold:
\begin{align*}
P_{L, K, \beta, h, r}  \big ( m_{\{o,z\}}^i  > M   \big )  & \leq 
P_{L, K, \beta, h, r}  \big ( n_o- m_{\{o,g\}}^2   > M  \big ) \\
& \leq 
\sum\limits_{p=M+1}^{\infty}
P_{L, K, \beta, h, r}  \big ( n_o  - m_{\{o,g\}}^2 = p  \big ) \\
& = 
\sum\limits_{p=M+1}^{\infty}  \langle    \,  \Psi_{p,o} \,  \big  \rangle \\
& \leq \sum\limits_{p=M+1}^{\infty}  \Big (  \langle  \prod_{x \in V}   \Psi_{p,x} \big  \rangle \Big )^{\frac{1}{|V|}} \\
& = \sum\limits_{p=M+1}^{\infty} P_{L, K, \beta, h, r} \big (   \forall x \in V :   n_x - m_{\{x,g\}}^2 = p \big )^{\frac{1}{|V|}}  \\
& \leq   
  \sum\limits_{p=M+1}^{\infty} 
  \Big (  \frac{1}{Z^{path}_{L,K, \beta, h, r}   } d^{|V|}  {u(p)}^{ |V|/8   }
  \, \, e^{  2 \,  (|E| + |E_g| ) \max{\{\beta, 1\}}}
  \Big )^{ \frac{1}{|V|}}
 \\
& \leq   c  \sum\limits_{p=M+1}^{\infty} { u(p)}^{\frac{1}{8}}.
\end{align*}
For the first inequality we used that, by definition of local time,  Definition \ref{deflocaltime}, 
for $z \in V$ neighbour of $o \in V$, we have that 
$m^{i}_{o,z} \leq n_o - m^2_{o,g}$;
for the first identity we used  Proposition \ref{prop:conversionlocaltime};
for the third inequality we use Proposition \ref{prop:chessboardabstract}.
For the fourth inequality we used a crude upper bound for the probability of the event 
$\{ \forall x \in V :   n_x - m_{\{x,g\}}^2 = p \}$.
The upper bound is as follows:
for each $x$, since 
 $n_x -  m_{\{x,g\}}^2  =p$,    
then, by definition of local time, 
there exists one of the $d$ edges of $E$ incident to  $x$ with at least 
$\lfloor p/d \rfloor$ dimers of colour $2$
(say,  an `high occupancy' edge).
Observe that there are at most $d^{V}$ possibilities
for choosing a configuration of high occupancy edges
and that there are at least $\lfloor \, |V|/8 \, \rfloor$ high occupancy  edges
(at least one for each even vertex).
Once one of  configuration of high occupancy edges is fixed,  we sum
over the number of dimers  independently for each
edge and for each colour, ignoring the constraint
that the number of blue dimers equals the number of red dimers at each vertex,
 using that $U(n) \leq 1$ for each $n \in \mathbb{N}_0$,
 and factorising the sum over the matchings with the factor $\frac{1}{n_x(w)!}$ 
 in the definition of the measure, Definition \ref{defmeasure}. 
We then obtain for each edge in $E$ which is not high occupancy and for each colour 
a multiplicative factor $e^{\beta}$,  for each edge in $E_g$ which is not high occupancy a multiplicative factor $e^{h \beta}$,
and for each edge which is high occupancy (and is then in $E$) and for the colour $2$ a multiplicative factor  at most 
\begin{equation}\label{eq:updefinition}
u(p) := \sum\limits_{ \ell = \lfloor p/d \rfloor   }   \frac{  \beta^\ell   }{\ell!},
\end{equation}
thus giving the desired inequality. 
For the last inequality we used used (\ref{eq:lowerboundZ}).
The proof then follows by observing that 
 $\sum_{p=M}^{\infty} {u(p)}^{\frac{1}{8}}$ converges to zero with $M$.

For (\ref{eq:mestimates2}) we observe that if $U(0)=0$ then the proof is obvious by definition of the measure,
and by the fact that $U(n)>0$ for some $n>0$ by assumption. 
   Hence,  we can assume that $U(0)>0$.  We  have that for some $c \in (0,1)$ we have 
\begin{align*}
P_{L, K, \beta, h, r} \big (  n_o= 0 \big )  & = 
  \langle   \Psi_o   \big  \rangle  \\
  & \leq 
  \Big ( 
    \langle  \prod_{x \in V}   \Psi_x   \big  \rangle
    \Big )^{\frac{1}{|V}|}
     \\ 
    & = 
      \Big (    
        P_{L, K, \beta, h, r} 
      \big ( \forall x \in V,   n_x= 0 \big )   \Big )^{\frac{1}{|V|}}
    \\ 
    & = 
      \Big (  \frac{U(0)^{|V|}}{ Z^{path}_{G, \beta, h, r   }  }\Big )^{\frac{1}{|V|}}  \leq c.
\end{align*}
where the first identity follows 
from  Proposition \ref{prop:conversionlocaltime},
the first inequality  follows from  Proposition \ref{prop:chessboardabstract},
and the second identity follows again from Proposition \ref{prop:conversionlocaltime},
the third identity follows from the definition of the measure  of the 
multi-occupancy double dimer model,  Definition \ref{defmeasure}.
In order to see that the last inequality holds for some $c<1$ we need to lower bound
 $Z^{path}_{G, \beta, h, r   }$.
 For this,  let $k>0$ be an integer such that $U(k)>0$. 
 We identify  a subset $A \subset E$ of edges 
 such that  each pair of edges in $A$ touch distinct vertices
 and such that  
 $|A| > \frac{1}{100} |V|$.
We then lower bound $Z^{path}_{G, \beta, h, r}$ by the weight of the configurations
$w \in \mathcal{W}^\ell$ such that each edge in  $A$ is either empty or it host precisely
$k$ blue and $k$ red dimers,
while all the remaining edges are empty. 
This leads to the trivial lower bound 
$$Z^{path}_{G, \beta, h, r}
 \geq   U(0)^{|V|- 2 |A|}  ( U(0)^2 +   \frac{\beta^{2k} U(k)^2}{k!^2 })^{{|A|}} \geq U(0)^{|V|} \big ( 1 +   \frac{\beta^{2k} U(k)^2}{k!^2 U(0)^2 } \big )^{  \frac{|V|}{100}  }.$$
 This leads to the desired bound and concludes the proof. 
\end{proof}

\begin{lem}
\label{lem:neighbourpoints}
 For any $\beta >0$, $h \in [0, 1]$,  $r \in \mathbb{N}$,  there exists a constant $c = c(\beta) \in (0, \infty)$ such that,
for any  $L, K \in 2 \mathbb{N} \cup \{1\}$,  for any $z \in V$ neighbouring $o \in V$ we have that 
\begin{align}
 \langle  S_o^i S_z^i \rangle_{L, K, \beta, h, r}    \,  \,  \,   &  =
\frac{1}{\beta}  E_{L, K, \beta, h, r}  \big ( \frac{1}{m_{\{o,z\}}^i}  \mathbbm{1}_{\{m_{\{o,z\}}^i>0\}} \big )  \leq \frac{1}{\beta}, \label{eq:neighbourcorr1} \\
 \sum\limits_{z \sim o}  \langle e^{ i ( s^i_o - s^i_z)  }   \rangle_{L, K, \beta, h, r}  & \geq c.
 \label{eq:neighbourcorr2}
\end{align}
\end{lem}
\begin{proof}
We start from the proof of the identity in the statement. 
From Proposition \ref{prop:conversion} we deduce that 
$$
\langle  S_o^i S_z^i \rangle_{L, K,  \beta, h, r} = 
\frac{ \mu_{\beta, h, r}( \mathcal{A}) }{Z^{path}_{L, K, \beta, h, r}},
$$
where $\mathcal{A}$ is the set of configurations with precisely one $i$-dimer on $\{o,s\}$, one $i$-dimer on $\{z,s\}$, and no further dimer touching $s$.
We introduce a map $f$ which takes any configuration $(m, \pi) \in \mathcal{A}$
and associates to it the configuration $f(m,\pi)  \in \mathcal{W}^\ell$ 
which is defined by following the next three steps:
\begin{enumerate}[(i)]
\item removing the $i$-dimers on 
$\{o,s\}$ and $\{z,s\}$,
\item  adding a  $i$-dimer on $\{o,z\}$,
\item pairing  the added dimer at $z$ and at $o$ 
to the dimers which were paired to those that have been removed.
\end{enumerate}
It follows from the definition that the map is a bijection from $\mathcal{A}$ to
$
f ( \mathcal{A} ) = \{  (m, \pi) \in \mathcal{W}^\ell \, : \, m^1_{\{o,z\}} > 0  \}.
$
For any configuration $w =(m, \pi) \in \mathcal{A}$,  $w^\prime  = (m^\prime, \pi^\prime)=   f(w)$,
we have that
\begin{equation}\label{eq:multivalue1}
\mu_{\beta, h, r}(w) =\frac{1}{\beta} \frac{1}{ (m^i_{\{o,z\}})^\prime} \mu_{\beta, h, r}(w^\prime) 
\end{equation}
From this we then deduce the identity  (\ref{eq:neighbourcorr1}).
The upper bound in (\ref{eq:neighbourcorr1}) then follows trivially. 

Let us now prove  (\ref{eq:neighbourcorr2}).
From Lemma \ref{lem:mestimates} and (\ref{eq:neighbourcorr1}) we deduce that 
\begin{multline*}
 \sum\limits_{z \sim o}  \langle e^{ i ( s^i_o - s^i_z)  }  \rangle  = 
  \sum\limits_{z \sim o}  
  \frac{1}{\beta} E \big ( \frac{1}{ m_{\{o,z\}}^i }  \mathbbm{1}_{  \{ m_{\{o,z\}}^i >0  \} }  \big ) 
  \geq   
  \sum\limits_{z \sim o}  
  \frac{1}{M \beta}  P\big (  0 < m_{\{o,z\}}^i , n_o \leq M  \big )  \\
  \geq  \sum\limits_{z \sim o}  
  \frac{1}{M \beta}  \Big (   P \big (  m_{\{o,z\}}^i   > 0  \big ) -   P \big (   n_o >  M  \big )   \Big ) 
  \\ \geq 
   \frac{1}{M \beta}  \Big (   P  \big (  n_o   > 0  \big ) -   6 P  \big (  n_o >  M  \big )   \Big )  \geq c.
 \end{multline*}
 where the last inequality follows from both inequalities in Lemma \ref{lem:mestimates}
  by choosing $M$ large enough depending on $\beta$.
\end{proof}

We now introduce an important quantity of interest.

\begin{defn}[Magnetisation] For any $\beta, h, r  \geq 0$ we introduce the \textit{magnetisation} 
$$
m_{G}(\beta, h, r) := \,  \langle   \cos( r s_o^1)  \rangle_{G, \beta, h, r}.
$$
\end{defn}
In the next proposition  we explore the limiting behaviour $|\S_{L,K}| \rightarrow \infty$ and $h \sim \frac{1}{|\S_{L,K}|} \rightarrow 0$ of this quantity.
In the first claim of the proposition we show  its important connection to the two point functions
which have been analysed in Proposition \ref{prop:twopoint} in the case  $h=0$. 
Thanks to (\ref{eq:firstinequalityr1}) below we will deduce an upper bound on the Cesaro sum of two point functions by bounding the magnetisation from above .
\begin{prop}[Magnetisation expansion] 
\label{prop:magnexp}
Set $\beta \geq 0$, $ r \in  \{1, 2\}$, $\ell \in \mathbb{N}$.
There exist a constant $c =c(\beta, \ell, r) \in (0, \infty)$  and a value $ h_0 >0$ small enough such that
for any  $\tilde h \in (0,  h_0)$,
 $L, K \in 2 \mathbb{N} \cup \{1\}$
\begin{align}
m_{\S_{L,K}}(\beta, h_{L,K}, r)  & \geq     \, \,  \frac{  \beta  \tilde h}{ | \S_{L,K}|}   \sum\limits_{x \in \S_{L,K}}   \langle   \cos(  r s_o^1)  \cos(  r  s_x^1)   \rangle_{L, K, \beta, 0, r}, \label{eq:firstinequalityr1} \\
    \label{eq:secondinequality}  
  \langle \cos ( \ell s_o^1)  \rangle_{L, K, \beta, h_{L,K}, r }    &   \leq  c  \,  \tilde h,
\end{align} 
where  $h_{L,K} =  \frac{\tilde h}{ |\S_{L,K}|}$.  
\end{prop}
\begin{proof}
From  a Taylor expansion of the exponentials in (\ref{eq:measure})  we deduce that
\begin{equation}\label{eq:startingpointexp}
 \langle   \cos(  \ell \, s_o^1)  \rangle_{L, K, \beta, h_{L,K}, r}  = 
\sum\limits_{ \substack{ m \in \mathbb{N}_0^V : \\ m \neq 0}  }   
\langle   \cos ( \ell s_o^1 )
 \prod_{x \in V}  \frac{ \big (2 \beta h_{L,K} \cos ( r s_x^1 ) \big )^{m_x}}{m_x!} 
\rangle_{L, K, \beta, 0, r}.
\end{equation}
For the lower bound in (\ref{eq:firstinequalityr1}) we 
use the positive-expectation condition  (recall Definition \ref{def:probabilitymeasure})
and reduce the sum in (\ref{eq:startingpointexp})
to the terms such that $m_x = 1 $ for just one $x \in V$
and $m_y=0$ for all $y \neq x$.
We then obtain (\ref{eq:firstinequalityr1}) by our choice of  $h_{L,K}$
and the condition $\ell = r$.
For  the upper bound in (\ref{eq:secondinequality}) we 
observe that there exists some $c = c(\beta, \ell, r) \in (1, \infty)$ such that,
 for each $m$ in the sum (\ref{eq:startingpointexp}),
 $$
 \langle   \cos ( \ell s_o^1 )
 \prod_{x \in V}  { \cos ( r s_x^1 )}^{m_x} \rangle_{L, K, \beta, 0, r}  \leq c \,  \prod_{x \in V} c^{ m_x  }.
 $$
 Indeed,   the function
\  $  \cos ( \ell s_o^1 ) \cos(r s_o^1)^{m_o}$
satisfies the assumptions in 
 Lemma  \ref{lem:chessboardbound}
 with $ R \leq  r m_o+ \ell $ and $M \leq 2^{m_o+1}$
 and the function 
$\cos(r s_x^1)^{m_x}$
satisfies the assumptions  in 
Lemma  \ref{lem:chessboardbound}
with $R \leq  r m_x $ and $M \,  \leq  \, 2^{m_x}$.
 We then obtain from  (\ref{eq:startingpointexp}) that 
$$
 \langle   \cos(  \ell s_o^1)  \rangle_{L, K, \beta, h_{L,K}, r} \leq  c 
\sum\limits_{ \substack{ m \in \mathbb{N}_0^V : \\ m \neq 0}  }   
 \prod_{x \in V}  \frac{ \big (\beta h_{L,K}  c \big )^{m_x}}{m_x!} 
 =c (e^{  c  \,  |V|  \,  \beta  \, h_{L,K}   }   - 1  ) \leq 
2 \,  \beta \, c^2 \tilde h,
$$
where in the last step we used that $\tilde h$ is small enough
depending on $\beta$, $r$ and $\ell$. 
\end{proof} 

\section{Bogoliubov-type inequality}
\label{sect:CauchySchwarz}
This section presents the central argument for our new proof of the Mermin--Wagner theorem for non positive (but reflection positive) measures.

We denote by $\Hcal :=   (-L/2, L/2] \cap (2 \mathbb{Z}+1)  e_1 $ and by  $\Vcal := (-L/2, L/2] \cap (2 \mathbb{Z}+1)e_2$, $\mathcal W:= (-\frac{K}{2}, \frac{K}{2}] \cap (2 \mathbb{Z}+1)  e_3 $   the subsets of $\S_{L,K}$ along the line passing through the origin and parallel to the Cartesian vectors. 
We introduce the 
Fourier dual torus,
$$
\S_{L,K}^* : =  \big \{  2\pi  (\frac{n_1}L, \frac{n_2}L,\frac{n_3}K) \in \mathbb{R}^3  \, \, : \, \,   n_i \in  ( - \frac{L}{2}  , \frac{L}{2} ] \cap \mathbb{Z}\, ,  \text{ for } i=1,2\, ,  n_3 \in  ( - \frac{K}{2}  , \frac{K}{2} ] \cap \mathbb{Z} \big \},
$$
and, given and fixed any Fourier mode  $k \in \S_{L,K}^*$,  for any
 two vectors $a, b \in \mathbb{C}^{\S_{L,K}}$,  we introduce the inner product,
$$
a \circ b :=  \sum\limits_{x \in \mathcal{H} \cup \mathcal{V} \cup \mathcal W} \cos (k \cdot x ) \, a_o \,  b_x  
 + a_o b_o
$$ 
We now define the \textit{extended Hamiltonian,}
$$
\tilde H := H + \sum\limits_{x \in \S^e_{L,k}}  \log (\gamma_x) +   \sum\limits_{x \in \S^o_{L,k}}  \log (\overline{\gamma_x}),
$$
and
the \textit{vector functions,}
$$
A :=  ( \sin (r s_x^1) )_{x \in \S_{L,K}} \quad \quad 
B :=  \beta \, \big (   \partial_{s_x^1} \tilde H -    \partial_{s_x^2} \tilde H  \big )_{x \in \S_{L,K}},
$$
and the quantity
\begin{align*}
p_z = p_z(L, K, \beta, h, r) : = \beta \big (  \langle  e^{ i (s_o^1 - s_{z}^1)  }  \rangle_{L, K, \beta, h, r} +  \langle e^{- i (s_o^2 - s_{z}^2 }  \rangle_{L,K, \beta, h, r} \big),
\end{align*}
for any  $z \sim o$. 
Standard computations based on integration by parts lead to the following central identities, which are one of the few results of our paper which rely in an essential way on the spin representation.
\begin{lem}[Identities]
\label{lem:identities}
For any $\beta,  h \geq 0 $, $r \in \{1,2\}$,  $L, K \in 2  \mathbb{N} \cup \{1\}$,  
we have that,
\begin{align}
\langle A \circ B \rangle_{L,K, \beta, h,  r} & =    \langle B \circ A  \rangle_{L,K, \beta, h,  r}  =  r  \, m_{\S_{L,K}}(\beta, h, r )  \\
\langle B  \circ B \rangle_{L,K, \beta, h,  r}  & =   
\sum\limits_{z \sim o}  p_{z}  \,   \big (  1 - \cos(k \cdot z) \big )    + 2 r^2 \, h \,  m_{\S_{L,K}}(\beta, h, r)
\end{align}
\end{lem}
\begin{proof}
From now on we will sometimes omit the subscripts from $\langle  \cdot \rangle $ to simplify the notation. 
For the first identity we note using  integration by parts and the periodicity of the function in the average to conclude that, 
$$
\langle  \sin(r s_o^1 )    \beta \, \big (   \partial_{s_x^1} \tilde H -    \partial_{s_x^2} \tilde H  \big )  \rangle =  r \langle   \cos( r s_o^1)  \rangle  \, \,  \delta_o(x).$$

For the second identity we use again integration by parts and the periodicity of the
function in the average to conclude that, 
\begin{align*}
\beta^2 \langle &  \big (   \partial_{s_o^1}  \tilde H -    \partial_{s_o^2} \tilde H  \big )     \,  \big (   \partial_{s_y^1} \tilde H -    \partial_{s_y^2} \tilde H  \big )  \rangle 
 = 
- \beta \langle \big (    \partial_{s_y^1}  -    \partial_{s_y^2}   \big )     \,  \big (   \partial_{s_o^1} \tilde H -    \partial_{s_o^2} \tilde H  \big )  \rangle  \\
&= 
\begin{cases}
- \beta \,  \langle   \partial^2_{s_o^1} \tilde H   +    \partial^2_{s_o^2} \tilde H  \big )  \rangle  &   \mbox{ if $y = o$}  \\
- \beta  \langle   \partial_{s_y^1}  \partial_{s_o^1}  \tilde H   +    \partial_{s_y^2} \partial_{s_o^2}  \tilde H  \big )  \rangle    &  \mbox{ if $y \sim o$}   \\
0    &   \mbox{ otherwise}   \\
\end{cases} \\
& = 
\begin{cases}
 \beta \,   \sum\limits_{y \sim o}     \Big ( \langle e^{ i  ( s_o^1 -  s_y^1)   }   \rangle
  + 
 \langle e^{-  i  ( s_o^2 -  s_y^2)   }  \rangle + 2hr^2 \langle   \cos( r s_o^1)  \rangle  \Big )   &   \mbox{ if $y = o$}  \\
- \beta  \big (  \langle e^{ i  ( s_o^1 -  s_y^1)   }  \rangle  + 
 \langle e^{-  i  ( s_o^2 -  s_y^2)   }  \rangle   \big )   &  \mbox{ if $y \sim o$}   \\
0    &   \mbox{ otherwise. }   \\
\end{cases} \\
\end{align*}
This concludes the proof. 
\end{proof}

 \subsection{Positivity of the Fourier transform }
We next state an inequality which is a discrete analogue of the fact that the Fourier transform of a convex, symmetric function on $[-L,L]$ is  nonnegative.   
The bound will be applied to two-point functions (\ref{eq:expressiontwopoint}).
Since our convexity property (\ref{eq:convexity}) fails at a neighbourhood of zero, we obtain a bound which may be negative, but which is  negligible comparing to the volume $|\S_{L,K}|$. 
We define for any $M \in \mathbb{N}$,  $\mathbb{Z}^M_{odd} :=  (2 \mathbb{Z} + 1 ) \cap  (- M/2, M/2]$
and 
$\mathbb{Z}^M_{even} :=  (2 \mathbb{Z}  ) \cap  (- M/2, M/2]$.
\begin{lem}
\label{l.IBP}
Let $G:   \S_{L,K} \to \R$ be such that
\begin{enumerate}[(i)]
\item  $G(x) = G(-x)$ for any $x \in e_1 \mathbb{Z}^L_{odd} \cup  e_2  \mathbb{Z}^L_{odd} \cup e_3 \mathbb{Z}^K_{odd}$,
\item $2G(x) - G(x-2 e_i) -G(x+2 e_i) \le 0$ for any $x \in e_i   \mathbb{Z}^L_{odd} \setminus \{\pm 1, \pm3\}$ if $i \in \{1,2\}$ or for any 
$x \in e_3  \mathbb{Z}^K_{odd} \setminus \{\pm 1, \pm 3\}$,
\item $G(  e_i  ) \geq  G(   3 e_i )$ for each $i \in \{1,2,3\}$.
\end{enumerate}
Then for every $k \in \S_{L,K}^*$
such that $k_j \not\in \{0, \pi\}$ for each $j \in \{1,2,3\}$ we have that 
\begin{align}\label{eq:identitysumxG1}
\sum_{x\in e_i \Z_{odd}^L} G(x) \cos (  k\cdot x  )
& \ge
-(G(3 e_i) -G(5 e_i)) \frac{1-\cos 3k_i}{2\sin^2 (k_i)}, \\
\sum_{x\in e_3 \Z_{odd}^K} G(x) \cos  (k\cdot x )
& \ge 
\label{eq:identitysumxG2}
-(G(3 e_3) -G(5 e_3)) \frac{1-\cos 3k_3}{2\sin^2 (k_3)},
\end{align}
where $i \in \{1,2\}$. 
\end{lem}

\begin{proof}
We note that for $i \in \{1,2\}$, 
\begin{align}
\sum_{x\in  e_i \Z^L_{odd}} G(x) \cos ( k\cdot x )
& = 
\frac1{2\sin^2 (k_i)} \sum_{x\in \Z^L_{odd}} (2G(x) - G(x-2 e_i) -G(x+2 e_i) ) \cos (k\cdot x) \\
\sum_{x\in  e_3 \Z^K_{odd}} G(x) \cos ( k\cdot x )
& = \label{eq:bypart2}
\frac1{2\sin^2 (k_3)} \sum_{x\in \Z^K_{odd}} (2G(x) - G(x-2 e_3) -G(x+2 e_3) ) \cos (k\cdot x).
\end{align}
To see this,  we notice that by the trigonometric identities
\begin{align*}
\sum_{x\in  e_i \Z^L_{odd}} G(x) \cos ( k\cdot x )
&= 
\sum_{x\in e_i  \Z^L_{odd}} G(x) \frac{\sin  \big (k\cdot(x+e_i) \big) - \sin  \big( k\cdot(x-e_i) \big) }{2\sin  (k_i )} \\
&=
\frac 1 {2\sin k_i}  \sum_{x\in e_i  \Z^L_{even}}  (G(x-e_i) -G(x+e_i) ) \sin  (k\cdot x ) \\
&= 
\frac 1 {2\sin k_i}  \sum_{x\in e_i  \Z^L_{even}} (G(x-e_i) -G(x+e_i) ) \frac{\cos
\big ( k\cdot(x+e_i) \big)- \cos  \big( k\cdot(x-e_i) \big )}{-2\sin ( k_i )}\\
&= 
\frac1{2\sin^2 (k_i)} \sum_{x\in e_i \Z^L_{odd}} (2G(x) - G(x-2e_i) -G(x+2e_i) ) \cos ( k\cdot x),
\end{align*}
and similar computations lead to (\ref{eq:bypart2}).
By our assumption (ii) we can lower bound the Fourier sum by replacing $\cos ( k\cdot x)$ by $1$ for any $x$ in the sum such that  $|x|_1>3$.  This leads to several cancellations in the sum and implies
\begin{align*}
\sum_{x\in e_i \Z^L_{odd}} G(x) \cos  (k\cdot x)
&
\ge
\frac1{2 \sin^2 (k_1)^2} \big ( 2(G(  5 e_i) -G( 3 e_i))  
 + 2(2 G(3 e_i) -G(e_i) -G(5 e_i)) \cos (3k_i)  \\& 
 + 2(G(e_i)-G(3 e_i)) \cos( k_i ) \big ) \\
&= 
\frac1{2 \sin^2  (k_1)}   \left(  (G(e_i) -G(3 e_i))(\cos (k_i) - \cos (3k_i))  
+( G(3 e_i) -G(5 e_i))(\cos (3k_i) -1) \right),
\end{align*}
where we also used the  condition (i).
The previous expression and (iii) together lead to (\ref{eq:identitysumxG1}) as desired. 
Using analogous computations we obtain   (\ref{eq:identitysumxG2})
and conclude the proof. 
\end{proof}

\subsection{Statement and proof of the Bogoliubov-type inequality}
We now state a Bogoliubov-type inequality for the 
quantities $A$ and $B$ which have been defined above. 
Our inequality holds for small external magnetic fields,  namely whose intensity  vanishes proportionally to the volume. 
It will turn  out that the error term, corresponding to the second term in the RHS of (\ref{eq:CS}) below,  is of smaller order if compared to the remaining terms. 

Define the positive cone of the dual torus 
$$
\S_{K,L,+}^*:=   \big \{  2\pi  (\frac{n_1}L, \frac{n_2}L,\frac{n_3}K) \in \mathbb{R}^3  \, \, : \, \,   n_i \in  ( - \frac{L}{4}  , \frac{L}{4} ] \cap \mathbb{Z}\, ,  \text{ for } i=1,2\, ,  n_3 \in  ( - \frac{K}{4}  , \frac{K}{4} ] \cap \mathbb{Z} \big \}
$$
\begin{thm}[Bogoliubov-type inequality]
\label{thm:cauchyschwarz}
For any $\beta >0$ and $r \in \{1,2\}$ there exists $\tilde h_0  >0$  and  $c  = c (\beta)\in \mathbb{R}$ such that, for any $\tilde h \in [0,  \tilde h_0)$,  $L,K \in 2 \mathbb{N} \cup \{1\}$   and 
$k \in \S_{K,L,+}^*$,
we have that,
\begin{align}
\label{eq:CS} 
\Big ( {\langle A\circ B\rangle_{L, K,  \beta, h_{L,K}, r } } \Big ) ^2 
\leq 
\big (  \langle A\circ A\rangle_{L, K,  \beta, h_{L,K},r }  + c \big )  \langle B\circ B\rangle_{L, K,  \beta, h_{L,K}, r }
\end{align}
where    $h_{L,K} :=  \frac{ \tilde {h}}{| \S_{L,K}|} $.
\end{thm}
\begin{proof}
Here and in the next proofs we omit the subscripts from $\langle  \, \, \, \rangle$ to simplify the notation.
Take $\mu \in\R$,  using bilinearity of the inner product,  
\begin{align*}
 \langle (\mu B- A)\circ (\mu B- A)\rangle
 &=
\sum_{y\in \Hcal} \cos  (k\cdot y)  \, \, \langle (\mu B_o - A_o)(\mu B_y -  A_y)\rangle\\
 &+ \sum_{y\in \Vcal} \cos  (k\cdot y)  \, \,  \langle (\mu B_o - A_o)(\mu B_y -  A_y)\rangle\\
 &+ \sum_{y\in \mathcal W} \cos  (k\cdot y)  \, \,  \langle (\mu B_o - A_o)(\mu B_y -  A_y)\rangle\\
& +  \langle (\mu B_o -  A_o)^2\rangle.
\end{align*}

For the first two terms on the right side above  we first note that 
 $( \mu B_x -   A_x )_{x \in \S_{L,K}}$
 is a reflection invariant vector of functions.  Hence by Propositions
 \ref{prop:convexity} and \ref{prop:monotonicity} the
  assumptions in Lemma \ref{l.IBP} are satisfied for the two point function $G(y) = \langle ( \mu B_o -   A_o)(\mu B_y -  A_y)\rangle$,  and we can then apply that lemma to conclude that 
\begin{align*}
\sum_{y\in \Hcal} \cos ( k\cdot y )  \langle (\mu B_o - A_o)(\mu B_y-   A_y)\rangle 
 +
\sum_{y\in \Vcal} \cos ( k\cdot y  )  \langle (\mu B_o -   A_o)(\mu B_y-   A_y)\rangle\\
+ \sum_{y\in \mathcal W} \cos ( k\cdot y  )  \langle (\mu B_o -   A_o)(\mu B_y-   A_y)\rangle\\
 \geq
\sum_{i=1}^3 (G(5e_i) -G(3e_i) ) \cdot \frac{1-\cos (3k_i)}{2\sin^2 ( k_i )}
 \geq -  c,
 \end{align*}
 by our assumptions on $k$.  Note that for the last inequality we observed that,  for any $x \neq o$, we have that 
$G(x) = \langle   \sin (rs_o^1)  \sin (rs_{x}^1)\rangle$,
and used the positive-expectation condition and Lemma \ref{lem:chessboardbound}
to upper bound 
$G(x) \leq $ $  \langle   \cos (rs_o^1)   \cos (rs_{x}^1) \rangle $ $ \leq c$
for any $x \in \S_{L,K}$
uniformly in $L$.

For the diagonal term,  using integration by parts for the second  term in the expansion of the square, we obtain that
 \begin{align*}
 \langle ( \mu B_o -   A_o)^2\rangle
 =  \langle \sin^2 (r s^1_o )\rangle
 -2 \, \mu \beta r\, m
 +  \mu^2  \beta^2 \langle (\partial_{s_o^1}\tilde  H   - \partial_{s_o^2} \tilde  H )^2\rangle.
  \end{align*}
Using $\sin^2 x = \frac{1}{2} - \frac{1}{2} \cos( 2x) $
we obtain that the discriminant $\Delta$  satisfies,
 \begin{align}\label{eq:discriminant}
 \frac 14 \Delta & = \beta^2  r^2m^2  \, - \,   \beta^2 \frac{1}{2}  \langle (\partial_{s_o^1}\tilde  H   - \partial_{s_o^2} \tilde  H )^2\rangle
 + 
  \beta^2 \frac{1}{2}  \langle \cos (2 r s^1_o ) \, \rangle \langle (\partial_{s_o^1}\tilde  H   - \partial_{s_o^2} \tilde  H )^2\rangle.
 \end{align}
Using again integration by parts and applying Proposition \ref{prop:magnexp} we obtain that,   for any $\beta>0$,  if $\tilde h>0$ is small enough (depending on $\beta$), then there exists $c = c(\beta, \tilde h)>0$ such that, for any $L , K  \in 2 \mathbb{N} \cup \{1\}$,
\begin{align*}
\beta^2 \langle (\partial_{s_o^1}\tilde  H   - \partial_{s_o^2} \tilde  H )^2\rangle &  =  
 - \beta  \, \langle    \, \partial_{s_o^1}^2 \tilde H   \, \rangle 
-   \beta  \, \langle    \, \partial_{s_o^2}^2 \tilde H   \, \rangle \\ &  = 
\beta   \sum\limits_{y \sim   o} \langle  e^{i (s_o^1 - s_y^1)} \rangle 
  + 
  \beta   \sum\limits_{y  \sim  o} \langle  e^{- i (s_o^2 - s_y^2)} \rangle
     +  2\beta hr^2 \langle   \cos( r s_o^1)  \rangle   \\ & \geq c,
\end{align*}
where for the last inequality we also used the lower bound in Lemma \ref{lem:neighbourpoints}  and  the positive-expectation for $\cos( r s_o^1)$.
Applying again \eqref{eq:secondinequality}   for the three terms in (\ref{eq:discriminant})  (which implies, in particular, $m \le c\tilde h$ and $\langle \cos( 2r s_o^1 ) \rangle) \le  c\tilde h$), 
we then deduce that,  under the same assumptions,  $ \frac 14 \Delta  \leq - \frac{c}{2} < 0$
where $c$ is the same constant as in the previous display.  Therefore, the diagonal term $\langle ( \mu B_o -   A_o)^2\rangle \ge 0$.

Combining the estimates for the diagonal and the non-diagonal terms we then obtain that, under the assumptions of our theorem,   
  \begin{equation*}
 \forall \mu \in \mathbb{R} \quad \langle (\mu B-  A)\circ (\mu B- A)\rangle \geq -c
 \end{equation*}
where $c = c (\beta)$ does not depend on $\mu$. 
 We conclude the proposition by taking $\mu = \frac{\langle A \circ B\rangle } {\langle B \circ B \rangle } $
 and using the fact that, by Lemma \ref{lem:identities}, $\langle A \circ B \rangle = 
 \langle B \circ A \rangle$ and that $\langle B \circ B \rangle \geq 0$. 
\end{proof}

\subsection{Estimates for $\langle A \circ A \rangle$}
In this section we provide some estimates for the sum over Fourier modes
of $\langle A\circ A\rangle_{L,K, \beta,h, r}$.
\begin{lem}
\label{lemm:upperA}
There exists $c = c( \beta) <\infty$ such that, for any $h \in [0,1]$, $r \in \{1,2\}$,  $K,L \in 2 \mathbb{N} \cup \{1\}$,  we have that
$$
\sum_{k\in \S_{L,K}^*} \langle A\circ A\rangle_{L,K, \beta,h, r}
\leq
c \,  K  \, L^2.
$$
\end{lem} 
\begin{proof}
We omit the subscripts in the expectation.  First note the trivial bound 
$$
\sum_{k\in \S_{L,K}^*}  \langle \sin^2 (rs_o^1)\rangle
\le \, 
| \S_{L,K}^*| \, \langle \sin^2 (rs_o^1)\rangle
\le 
cKL^2.
$$
Here we used the trigonometric identity $\sin^2 x = \frac 12- \frac 12 \cos (2x)$,  and $\langle \cos (2rs_o^1)\rangle \ge 0 $  by the positive-expectation condition.
We now upper bound the off-diagonal term.
For any $x \in \S_{L,K}$ we define the sets
 $\Hcal_x = \mathcal{H}+x$, 
 $\Vcal_x = \mathcal{V}+x$, 
 $\Wcal_x = \mathcal{W}+x$,
 where the sum is with respect to the torus metric,
 we also recall the definitions 
  $\mathbb{Z}^{M}_{odd} =     (2 \mathbb{Z} + 1 ) \cap  (- M/2, M/2]  $,  
  $\Z^{M}_{even} =   2 \mathbb{Z}  \cap  (- M/2, M/2]$,  
  and define
  $\Z^M:=    (- M/2, M/2] \cap \mathbb{Z}$
  for  any $M \in \mathbb{N}$.
 Using translation invariance we have that 
\begin{align*}
& K L^2\Biggl(\sum_{y \in \Hcal} \cos(  k \cdot y ) \langle \sin (rs_o^1) \sin (rs_y^1) \rangle  
 +
 \sum_{y \in \Vcal} \cos(  k \cdot y ) \langle \sin (rs_o^1) \sin (rs_y^1) \rangle
 + \sum_{y \in \mathcal W} \cos(  k \cdot y ) \langle \sin (rs_o^1) \sin (rs_y^1)\rangle \Biggr) \\
& = 
 \sum_{x\in \S_{L,K}} \sum_{y \in \Hcal_x}  \cos(  k \cdot (x-y) ) \langle \sin (rs_x^1) \sin (rs_y^1) \rangle
 +
 \sum_{x\in \S_{L,K}}  \sum_{y \in \Vcal_x} \cos(  k \cdot (x-y) ) \langle \sin (rs_x^1) \sin (rs_y^1) \rangle\\
 &+ \sum_{x\in \S_{L,K}}  \sum_{y \in \mathcal W_x} \cos(  k \cdot (x-y) ) \langle \sin (rs_x^1)\sin (rs_y^1) \rangle.
\end{align*} 
Now we see that the previous quantity is less or equal to 
\begin{align*}
& 2 K L\sum_{x=(- \frac L2 +1) e_1}^{
\frac L2 e_1} \sum_{ y \in \Z^L_{odd}e_1+x}  \cos(  k \cdot (x-y)) \langle \sin (rs_x^1)\sin (rs_y^1) \rangle \\
&+L^2 \sum_{x=(- \frac K2+1) e_3}^{
\frac  K  2 e_3} \sum_{ y \in \Z^K_{odd}e_3+x}  \cos(  k \cdot (x-y)) \langle \sin (rs_x^1)\sin (rs_y^1) \rangle
\\
& = 2 K  L (\hat G^1_{\mathbb{Z}^L}(k) - \hat G^1_{\mathbb{Z}^L_{even}}(k)- \hat G^1_{\Z^L_{odd}}(k))
+ L^2 (\hat G^3_{\mathbb{Z}^K}(k) - \hat G^3_{\mathbb{Z}^K_{even}}(k)- \hat G^3_{\mathbb{Z}^K_{odd}}(k)),
\end{align*}
where we defined 
 for each $i \in \{1, 2, 3\}$,  each $M \in \mathbb{N}$, 
  and  each set $A \in \{ \mathbb{Z}^{M}_{odd}$,  $\Z^{M}_{even},  \mathbb{Z}^M  \}$
  the quantity
\begin{align*}
\hat G^i_{A}(k)& := \sum_{x,y \in  e_i A } \cos(  k \cdot (x-y)) \langle \sin (rs_x^1)\sin (rs_y^1) \rangle.
\end{align*}
We now observe that 
\begin{multline*}
\sum_{k\in \S_{L,K}^*} \hat G^1_{\mathbb{Z}^L}(k)
= 
\sum_{k\in \S_{L,K}^*}  \sum_{x,y= (-\frac L2+1) e_1}^{\frac L2 e_1} \exp(i  k \cdot (x-y)) \langle \sin (rs_x^1)\sin (rs_y^1) \rangle \\
 = |\S_{L,K}| \sum_{x=(-\frac L2+1) e_1}^{\frac L2 e_1} \langle \sin^2 (rs_x^1) \rangle
= KL^3 \langle \sin^2 (rs_o^1) \rangle \le C K  L^3,
\end{multline*}
where for the second equality we used the Parseval identity and  for the inequality used the fact that  $  \langle \sin^2 (rs_o^1) \rangle \leq \frac{1}{2}$, as observed before. 
Similarly,  we deduce that  
$\sum_{k\in\S_{L,K}^*} \hat G^1_{\mathbb{Z}^L_{odd}}(k)\le C K  L^3$,  $\sum_{k\in \S_{L,K}^*} \hat G^1_{\mathbb{Z}^L_{even}}(k)\le C K L^3$,   
$\sum_{k\in \S_{L,K}^*} \hat G^3_{\mathbb{Z}^K}(k)\le C K^2 L^2$,   
$\sum_{k\in \S_{L,K}^*} \hat G^3_{\mathbb{Z}^K_{odd}}(k)\le CK^2L^2$, \\ 
$\sum_{k\in \S_{L,K}^*} \hat G_{\mathbb{Z}^K_{even}}(k)\leq C K^2L^2$.
Putting together all the previous bounds we obtain that 
\begin{align*}
\sum_{k\in \S_{L,K}^*} \biggl(\sum_{y \in \Hcal} \cos(  k \cdot y ) \langle \sin (rs_o^1) \sin (rs_y^1) \rangle
&  + 
 \sum_{y \in \Vcal} \cos(  k \cdot y ) \langle \sin (rs_o^1) \sin (rs_y^1) \rangle
\\  & + \sum_{y \in \mathcal W} \cos(  k \cdot y ) \langle \sin (rs_o^1) \sin (rs_y^1) \rangle \biggr)  \leq 
 C'K L^2,
 \end{align*}
 and we conclude the lemma. 
\end{proof}

We also need a lower bound for the large momentum region.  

\begin{lem}
\label{lem:lowerA}
There exists $c=c(\beta)<\infty$ and $h_0>0$ 
such that, for any $\tilde h \in [0,h_0]$, $r \in \{1,2\}$,  $K,L \in 2 \mathbb{N} \cup \{1\}$,  
$$
\sum_{k\in \S_{L,K}^*\setminus \S_{L,K,+}^*} \langle A\circ A\rangle_{L,K, \beta, h_{L,K}, r} 
\geq
- c \, K  \, L^2
$$
where  $h_{L,K} =  \frac{\tilde h}{| \S_{L,K}|}$.
\end{lem} 
\begin{proof}
 Recall that 
\begin{multline}\label{eq:firstdisplaylemma}
\langle A\circ A\rangle
 =
  \sum_{y \in \Z^L_{odd} e_1} \cos(  k \cdot y) \langle \sin (rs_o^1)\sin (rs_y^1) \rangle + 
  \sum_{y \in \Z^L_{odd}e_2} \cos(  k \cdot y) \langle \sin (rs_o^1)\sin (rs_y^1) \rangle\\
  +  \sum_{y \in \Z^K_{odd}e_3} \cos(  k \cdot y) \langle \sin (rs_o^1)\sin (rs_y^1) \rangle+
  \langle \sin^2 (rs_o^1)  \rangle.
  \end{multline}
To begin, we  sum over $k$ and obtain for  any $y\in\Z$, $i \in \{1,2\}$, 
  \begin{align}
  \label{e.sumk1}
  \sum_{k\in \S_{L,K}^*\setminus \S_{L,K,+}^*}
  \cos (  k_i \,  y)
  =
  \sum_{n_3 \in ((-\frac K2,  \frac K2] \setminus(-\frac K4,  \frac K4])\cap\Z }
  \sum_{n_1, n_2 \in ((-\frac L2,  \frac L2] \setminus(-\frac L4,  \frac L4])\cap\Z }
    \cos ( \frac{2\pi n_1}L y) \notag \\
 =
\frac 14  KL  \sum_{n_1\in ((-\frac L2,  \frac L2] \setminus(-\frac L4,  \frac L4])\cap\Z }
    \cos ( \frac{2\pi n_1}L y)
    \end{align}
   Note that 
      \begin{align*}
     \sum_{n_1\in ([-\frac L2,  \frac L2] \setminus[-\frac L4,  \frac L4])\cap\Z }
    \cos ( \frac{2\pi n_1}L y)
    &= 
     \sum_{n_1\in ([-\frac L2,  \frac L2] \setminus[-\frac L4,  \frac L4])\cap\Z }
    e^{i\frac{2\pi n_1}L y} \\
    & 
    =
    e^{i \frac \pi 2 y } e^{i\frac{2\pi }L y}\frac{1-e^{i \frac \pi 2 y }}{1- e^{i\frac{2\pi }L y}}
    +
     e^{-i \frac \pi 2 y } e^{-i\frac{2\pi }L y}\frac{1-e^{-i \frac \pi 2 y }}{1- e^{-i\frac{2\pi }L y}} \\
     &= 
     \frac{i^y(1-i^y) (e^{i\frac{2\pi }L y} -1) +(-i)^y(1-(-i)^y) (e^{-i\frac{2\pi }L y} -1) }
     {2 -2 \cos  ( \frac{2\pi }L y )  }
    \end{align*}
    Since $y$ is odd,  we use $i^{2y} =-1$ to simplify the summation above as 
    \begin{align*}
     \sum_{n_1\in ([-\frac L2,  \frac L2] \setminus[-\frac L4,  \frac L4])\cap\Z }
    \cos ( \frac{2\pi n_1}L y)
    = 
    \frac{i^{y+1} \sin  ( \frac{2\pi}L y)}{1- \cos ( \frac{2\pi }L y)} -1.
     \end{align*}
     By comparing  the previous expression with \eqref{e.sumk1}
     we see that there  are mismatch terms at $n_1 =- \frac L2$ and  $n_1 =- \frac L4$, concluding that
      \begin{equation}\label{eq:sumoverk}
       \sum_{k\in \S_{L,K}^*\setminus \S_{L,K,+}^*}
  \cos (  k_i y)
  =
 \frac 14 KL  \frac{i^{y+1} \sin  ( \frac{2\pi}L y)}{1- \cos ( \frac{2\pi }L y)}  +O(KL),
   \end{equation}
   where the last term does not depend on $y$.
We can now lower  bound the sum over $k$ of    $\langle A\circ A\rangle$
to prove the lemma. 
 By \eqref{eq:secondinequality},  we see that $ \langle\sin^2 (rs_o^1)   \rangle = \frac 12 - \frac 12 \langle \cos (2rs_o^1)  \rangle \ge \frac 12 - ch_0$.  This implies, $ \langle \sin^2 (rs_o^1)  \rangle>0$ by taking $h_0$ sufficiently small.   Using this observation, (\ref{eq:firstdisplaylemma})  and  (\ref{eq:sumoverk})    we obtain that
     \begin{align}
     \label{e.LBalt}
  \sum_{k\in \S_{L,K}^*\setminus  \S_{L,K,+}^*}  \langle A\circ A\rangle
&  \geq
   \sum_{y \in \Z^L_{odd}e_1} \frac 14 KL  \frac{i^{y+1} \sin  ( \frac{2\pi}L y)}{1- \cos ( \frac{2\pi }L y)}  \langle \sin (rs_o^1)\sin (rs_y^1) \rangle  \notag \\
  & + 
  \sum_{y \in \Z^L_{odd}e_2} \frac 14 KL  \frac{i^{y+1} \sin  ( \frac{2\pi}L y)}{1- \cos ( \frac{2\pi }L y)}  \langle \sin (rs_o^1)\sin (rs_y^1) \rangle \notag \\
  &+ \sum_{k\in \S_{L,K}^*\setminus  \S_{L,K,+}^*}   \sum_{y \in \Z^K_{odd}e_3} \cos(  k \cdot y) \langle \sin (rs_o^1)\sin (rs_y^1) \rangle  \notag \\
 &- c K L \Biggl(  \sum_{y \in \Z^L_{odd}e_1}| \langle \sin (rs_o^1)\sin (rs_y^1) \rangle|  + 
  \sum_{y \in \Z^L_{odd}e_2}  |\langle \sin (rs_o^1)\sin (rs_y^1) \rangle| \Biggr).
   \end{align}
  Since $\{\frac{i^{y+1} \sin  ( \frac{2\pi}L y)}{1- \cos ( \frac{2\pi }L y)} \}$ is an alternating sequence with decreasing magnitude in  $(0, \pi) \cap (2 \mathbb{N}+1)$,  together with the monotonicity of the two point function,  $\langle \sin (rs_o^1) \sin (rs_{(2n-1)e_i}^1) \rangle \geq $ $\langle \sin (rs_o^1) \sin (rs_{(2n+1)e_i}^1) \rangle$ for odd integers  $2n-1 \in (1, L/2)$,
  and symmetry
  we see that the first two sums on the right side of \eqref{e.LBalt} are  bounded from below by the sum of the $y=e_1$ and   $y=-e_1$ term:
   \begin{align}
  \ \sum_{y \in \Z^L_{odd}e_1}  \frac 14 KL  \frac{i^{y+1} \sin  ( \frac{2\pi}L y)}{1- \cos ( \frac{2\pi }L y)}  \langle \sin (rs_o^1)\sin (rs_y^1) \rangle + 
  \sum_{y \in \Z^L_{odd}e_2}  \frac 14 KL  \frac{i^{y+1} \sin  ( \frac{2\pi}L y)}{1- \cos ( \frac{2\pi }L y)}  \langle \sin (rs_o^1)\sin (rs_y^1) \rangle  \notag \\
  \geq
  -cKL \frac L {2\pi  } (\langle \sin (rs_o^1) \sin (rs_{e_1}^1) \rangle + \langle \sin (rs_o^1) \sin (rs_{e_2}^1) \rangle) 
  \geq
  - c KL^2.
  \label{e.display1}
  \end{align}
  On the other hand,  using the monotonicity of the two point function we have 
   \begin{multline}
    - c K L \Biggl(  \sum_{y \in \Z^L_{odd}e_1} \langle \sin (rs_o^1)\sin (rs_y^1) \rangle + 
  \sum_{y \in \Z^L_{odd}e_2}  \langle \sin (rs_o^1)\sin (rs_y^1) \rangle  \Biggr)   \\
  \ge
  - c KL^2 (\langle \sin (rs_o^1) \sin (rs_{e_1}^1) \rangle + \langle \sin (rs_o^1) \sin (rs_{e_2}^1) \rangle)
  \ge 
  -c^\prime KL^2,
  \label{e.display2}
     \end{multline}
     where for the last inequality we use Lemma \ref{lem:chessboardbound} to deduce that $\langle \sin (rs_o^1) \sin (rs_{e_1}^1) \rangle   \leq c^2$ uniformly in $L$ and $K$.
     
     For the sum of $y$ in $e_3$ direction,  the same computation as  \eqref{e.sumk1} and \eqref{eq:sumoverk} yields
     $$
      \sum_{k\in \S_{L,K}^*\setminus \S_{L,K,+}^*}
  \cos (  k_3 \,  y)
 =
\frac 14  L^2  \sum_{n_3\in ((-\frac K2,  \frac K2] \setminus(-\frac K4,  \frac K4])\cap\Z }
    \cos ( \frac{2\pi n_3}K y)
    =
     \frac 14 L^2  \frac{i^{y+1} \sin  ( \frac{2\pi}K y)}{1- \cos ( \frac{2\pi }K y)}  +O(L^2),
     $$
     where the last term in the RHS does not depend on $K$ or $y$.
     Using that $\{\frac{i^{y+1} \sin  ( \frac{2\pi}K y)}{1- \cos ( \frac{2\pi }K y)} \}$ is an sequence series with decreasing magnitude in  $(0, K/2) \cap (2 \mathbb{N}+1)$,   the monotonicity of the two point function,  symmetry and Lemma \ref{lem:chessboardbound}, as before, we obtain that
     \begin{multline*}
      \sum_{k\in \S_{L,K}^*\setminus  \S_{L,K,+}^*}   \sum_{y \in \Z^K_{odd}e_3} |\cos(  k \cdot y) \langle \sin (rs_o^1)\sin (rs_y^1) \rangle | \\
     \ge
     -cL^2 \frac K {2\pi  } \langle \sin (rs_o^1) \sin (rs_{e_3}^1) \rangle 
     -cKL^2  \langle \sin (rs_o^1) \sin (rs_{e_3}^1) \rangle 
     \ge
     -c'KL^2.
     \end{multline*}
     Combing with \eqref{e.display1} and   \eqref{e.display2} above we conclude the proof of the lemma.  
\end{proof}

\section{Mermin--Wagner Theorem}
\label{sect:merminwagner}

The Mermin–Wagner theorem states that in two dimensions the magnetisation vanishes when the external field is taken to zero, independently of the inverse temperature.
This holds for spin systems with continuous symmetries, but not, for instance, for the Ising model in $d>1$ at low temperature.
We now state our analogue of the Mermin–Wagner theorem for our general framework.
In contrast to the classical setting, where one first takes the thermodynamic limit and then sends the external field to zero,  here, for technical reasons, we must let the field go to zero and the volume diverge simultaneously, with a suitable scaling.

\begin{thm}[Mermin--Wagner theorem]
\label{thm:magnetisationbound}
Suppose that $\beta \geq 0$,  $r \in \{1,2\}$, and that $\tilde h>0$ is  small enough.  There  a constant $c = c (\beta) <\infty$  such that, for any    $L, K \in 2 \mathbb{N} \cup \{1\}$ satisfying $K\le (\log(L))^{\frac 12}$, we have 
$$
m_{\S_{L,K}}(\beta, h_{L,K}, r)  \leq  c \sqrt{  \frac{ K}{\log(L)}},
$$
  where  $h_{L,K} := \frac{\tilde h }{|\S_{L,K}|}$.  
\end{thm}
\begin{proof}
From  (\ref{eq:CS})  and Lemma \ref{lem:identities} we deduce that
\begin{multline*}
 \frac{r^2 \,  m_{\S_{L,K}}(\beta, h_{L,K}, r)^2}{ \frac{2}{\beta} |k|^2  +8 h_{L,K} \,  m_{\S_{L,K}}(\beta, h_{L,K}, r) } \\  \leq \frac{r^2 \,  m_{\S_{L,K}}(\beta, h_{L,K}, r)^2}{ \sum\limits_{z \sim o}  p_{z}  \,   \big (  1 - \cos(k \cdot z) \big )    + 2 r^2 \, h_{L,K} \,  m_{\S_{L,K}}(\beta, h_{L,K}, r) }  \\ \leq \big ( \langle A\circ A\rangle_{L, K, \beta, h_{L,K}, r} +  c  \big ).
\end{multline*}
In the sequel we omit the dependences of $m_{\S_{L,K}}$, which continues to depend on $\beta$ and $h_{L,K}$,  and $r$, in order to simplify the notation. 

Summing over $k \in  \S_{L,K,+}^*$, applying Lemma \ref{lemm:upperA} and \ref{lem:lowerA} to upper bound the RHS of the previous expression, and dividing  by $|\S_{L,K}|$  leads to
\begin{align}
\label{e.sumk}
\frac{1}{|\S_{L,K}|} \sum_{k\in  \S_{L,K,+}^*} \frac{m^2}{\frac{2}{\beta} |k|^2 + \frac{8\tilde h}{|\S_{L,K}|} m} 
\leq  c,
\end{align}
which holds for any $L, K\in 2 \mathbb{N} \cup 1$.

Let us now define the set $ \mathcal{U} :=   \{ (L, K) \in  (2 \mathbb{N}+1)^2 \, : \, K \leq \sqrt{ \log(L) }   \}$.
Suppose  that for any $\varphi>0$ we have that
 \begin{equation}\label{eq:contradictionfinding}
m_{\S_{L,K}}(\beta, h_{L,K}, r)  \geq \varphi \sqrt{ \frac{K}{\log(L)}} \, \, 
 \mbox{ for infinitely many  $(L, K) \in \mathcal{U}$.}
 \end{equation}
 We will find a contradiction with this assumption, deducing that 
 there exists a value $\varphi > 0$ such that 
 $m \geq \varphi \sqrt{ \frac{K}{\log(L)}}$
 can only hold for finitely many choices of $(L,K) \in \mathcal{U}$.
 This implies the claim of the proposition, as desired.

Let us now show that (\ref{eq:contradictionfinding}) leads to a contradiction.
 From (\ref{e.sumk}) and  (\ref{eq:contradictionfinding})   we deduce that, for infinitely many choices of $(L,K) \in \mathcal{U}$, 
\begin{equation}\label{eq:contra1}
\frac{1}{|\S_{L,K}|} \sum_{k\in  \S_{L,K,+}^*} \frac{\varphi^2}{c_2 |k|^2 + \frac{c_1}{|\S_{L,K}|}}  
\leq   c_3  \frac{ \log(L)}{K},
\end{equation}
where we also applied Proposition \ref{prop:magnexp}
in order to upper bound $m$ in the denominator. 
We now lower bound the term in the LHS  of (\ref{eq:contra1}) from the term below,
$$
\frac{1}{|\S_{L,K}|} \sum_{ \substack{ k\in\S_{L,K,+}^*: \\ |k|^2 \geq  \frac{c_1}{  c_2 }\frac{1}{| \S_{L,K}|}}} \frac{\varphi^2}{ 2 c_2 |k|^2} +
 \sum_{ \substack{k\in  \S_{L,K,+}^*: \\ |k|^2 < \frac{c_1}{ c_2} \frac{1}{| \S_{L,K}|}}} \frac{\varphi^2}{2 c_1 }.
$$
The main contribution comes from the first sum,
which, after reducing the sum to all $k$ such that  $k_3 = 0$
and using the change of variable  $j_i=  \frac{ L k_i}{2 \pi}$ for $i=1,2$,
equals
$$
\frac{1}{ 8 \pi^2 c_2 } \frac 1K \sum_{ \substack{ j_1, j_2 \in \{- \frac{L}{4}+1,  -\frac{L}{4} + 2,  \ldots \frac{L}{4} \} : \\ j_1^2 + j_2^2 \geq c / K }  } \frac{\varphi^2}{ (j_1^2+j_2^2)}.
$$
We finally note that the previous term is bounded from below by  $c  \, \varphi^2 \,  \frac{1}{K} \, \log (L) $.  
Summarizing,  we deduce that for any $\varphi > 0$, and 
 infinitely many values $(L, K) \in \mathcal{U}$
\begin{equation}
\label{e.contra}
c_4 \varphi^2  \frac{\log (L)}{K} \le c_3 \frac{\log(L)}{K},
\end{equation}
which is not possible.
Hence,  we obtained the desired contradiction and concluded the proof. 
\end{proof}

The next theorem shows how the decay of magnetisation implies the decay of the two–point functions.\begin{thm}[From magnetisation decay to correlations decay]
\label{thm:maintheoremgeneralloop}
Let $\beta \geq 0$,  $\ell   \in \{1,2\}$ be arbitrary.  There exists $c = c(\beta) < \infty$
 such that for any $L, K  \in 2 \mathbb{N} \cup \{1\}$ satisfying $K \leq  \sqrt{\log (L)}$ we have 
\begin{align}
\sum_{x \in \S_{L,K}^o} \frac{  \mathcal{G}_{L, K, \beta}^{(1)}(o,x)}{|\S^o_{L,K}|}  &  \leq  c \sqrt{ \frac{K}{\log(L)}} \label{eq:generalclaim1}\\
\sum_{x \in \S_{L,K}} \frac{  \mathcal{G}_{L, K, \beta}^{(\ell)}(o,x)}{|\S_{L,K}|}  & \leq 
c \sqrt{ \frac{K}{\log(L)}},\label{eq:generalclaim}
\end{align}
where the second inequality holds only if $U$ is normalized monotone when $\ell=1$ and only if $U$ is fast decaying when $\ell=2$. 
\end{thm}

We may conclude the proof of the Mermin Wagner theorem.

\begin{proof}[Proof of  Theorems 
 \ref{thm:maintheoremdimer},
 \ref{thm:maintheoremspin},
and \ref{thm:maintheoremgeneralloop}]
Fix $\beta >0$, $ \tilde{h}  > 0$ small enough. 
We apply the first claim in Proposition \ref{prop:magnexp} 
and Theorem \ref{thm:magnetisationbound} to deduce our main theorem for the spin system,  Theorem \ref{thm:maintheoremspin}.

We now apply Proposition \ref{prop:twopoint}
and the observation
$\sum_{x \in \S_{L,K}}  \mathcal{G}_{L, K, \beta}^{(2)}(o,x) =
2 \sum_{x \in \S_{L,K}^o}   \mathcal{G}_{L, K, \beta}^{(2)}(o,x) $
to deduce from Theorem  \ref{thm:maintheoremspin}  the proof of Theorem \ref{thm:maintheoremgeneralloop}.

We now use our  Theorem \ref{thm:maintheoremgeneralloop}
to deduce our main theorems on the dimer and monomer double-dimer model, respectively
Theorems  \ref{thm:maintheoremdimer} and   \ref{thm:maintheoremddmodel}.

For Theorem \ref{thm:maintheoremdimer}, 
recall that if  we choose  
$U= \mathbbm{1}_{\{n=1\}}$
then the  two-point function $\mathcal{G}^{(1)}_{L, K, \beta}(o,x)$
corresponds to  the monomer-monomer correlation 
of the dimer model, $\mathcal{C}(o,x)$,
as we observed in (\ref{eq:monomerrelation}).
Hence,  the first inequality in Theorem \ref{thm:maintheoremgeneralloop},
together with the observation that $\mathcal{C}(o,x)=0$ if $x \in \S_{L,K}^e$,
implies directly 
Theorem \ref{thm:maintheoremdimer}, as desired. 

We now prove Theorem \ref{thm:maintheoremddmodel}.
Note that  the second claim,  \eqref{eq:firsttheo2}, follows from 
the first claim,   \eqref{eq:firsttheo},  by a direct application of the Markov inequality.
We choose $U(0) = U(1)=1$ and $U(n) = 0$ for all $n >1$,
giving the monomer-double dimer model with monomer activity $\rho = 1/\beta$ when $\beta >0$, as we observed in Section \ref{s.special}.
Since under this choice $U$ is fast decaying (and normalized monotone),  
we can apply  Theorem \ref{thm:maintheoremgeneralloop} with $\ell =2$ 
and $\ell=1$ and
 (\ref{eq:monomerdimerloop}) 
to deduce  \eqref{eq:firsttheo} and \eqref{eq:firsttheo3} in  Theorem  \ref{thm:maintheoremddmodel}  when 
$\rho = 1/\beta >0$. 

When $\rho = 0$, \eqref{eq:firsttheo3} in Theorem  \ref{thm:maintheoremddmodel}
follows directly from Theorem \ref{thm:maintheoremdimer} since
in this case $\mathbb{G}_{L,K, 0} (o,x) = \mathcal{C}(o,x)$,
as we observed in Section \ref{s.special}.

In order to obtain \eqref{eq:firsttheo} in Theorem  \ref{thm:maintheoremddmodel} when  $\rho =0$
we use that 
$$
\mathbb{P}_{L, K,  0}( x \leftrightarrow y)  = 
\sum\limits_{z \sim x,  k \sim y}
 \,  \mathcal{C}(x,y) \mathcal{C}(z,k) \leq d^2  \,  \mathcal{C}(x,y),
$$
where the identity follows from the fact that, by \cite[Figure 2]{Kenyonclaire},
  $\mathcal{C}(x,y) \mathcal{C}(z,k)$ corresponds to the probability
that in the double-dimer model a loop connects $x$ to $y$ 
with the blue dimers of such loop touching $x$ and $y$ being on the edges $\{x,z\}$
and $\{y,k\}$, while the inequality follows from the fact that $\mathcal{C}(z,k) \leq 1$ for any $z, k \in \S_{L,K}$ by \cite[equation (2.25)]{LeesTaggiCMP2020}.  
Hence using the previous inequality and Theorem  \ref{thm:maintheoremdimer}
we deduce that 
\begin{equation}
\sum\limits_{x \in \mathbb{S}_{L, K}} \frac{ \mathbb{P}_{L, K,  0}( o \leftrightarrow x)  }{|  \mathbb{S}_{L, K} |} \leq cd^2 \sqrt{ \frac{K}{\log(L)}},
\end{equation}
from which we deduce  \eqref{eq:firsttheo}.

\end{proof}

\section{Appendix}
\label{sect:appendix}
\subsection{Multiple-type dimer model}
\label{sect:generalN}
We now explain how to   extend our framework to the case of dimers of $2N$ different types,
where $N$ is an integer.
The local spin space is the set 
$\Xi = [0, 2 \pi)^{2N}.$
The configuration space is the set  $\Omega_{s} = \Xi^V$, with 
 $\boldsymbol{s} = (s_z)_{z \in V}  \in \Omega_s$,
and  $s_z = (s_z^1,  \ldots, s_z^{2N})$. The spin variable at $z$ is then a vector
$$
S_z = (S_z^1,  \ldots, S_z^{2N}),
$$
where the function $S_z^k : \Omega_s \mapsto \mathbb{C}$ is defined for each $k \in \{1,\ldots,  2N \}$ as 
\begin{equation}\label{eq:spinvariableN}
S^k_z (\boldsymbol{s}) :=
\begin{cases}
e^{ i s_z^k  }   & \mbox{ if $k$ odd and $z$ even or $k$ is even and $z$ odd,}\\
e^{ - i s_z^k  } &   \mbox{ if $k$ odd and $z$ odd or $k$ is even  and $z$ even,}
\end{cases}
\end{equation}
We define for each $\boldsymbol{s} \in \Omega_s$ 
the function
\begin{equation}\label{eq:gfunctionN}
\gamma_z ( \boldsymbol{s}) : = \frac{1}{(2 \pi)^{2 N}}   \sum\limits_{n_1, \ldots, n_N \geq 0 } ^{\infty} U(n_1, \ldots, n_N)  
\prod_{i=1}^{N}
\big (\overline{ S_z^{2i-1}( \boldsymbol{s})  \,  \, S_z^{2i}( \boldsymbol{s}) } \big )^{n_i},
\end{equation}
 $U : \mathbb{N}_0^N \mapsto \mathbb{R}_0^+$ generalizes the weight function,
 and, as before, 
$
 \boldsymbol{\gamma}( \boldsymbol{s}) : = \prod_{z \in V}  \gamma_z(\boldsymbol{s}).
$
For each $\boldsymbol{s}  \in \Omega_s$
the Hamiltonian function, $H : \Omega_s \mapsto \mathbb{C}$, is  defined as 
$$
 H  (\boldsymbol{s})  =    
\sum\limits_{i=1}^{2N}  \sum\limits_{  \{x, y\} \in E  } 
S_{x}^i(\boldsymbol{s})  S_y^i(\boldsymbol{s}) + 
 h \sum\limits_{z \in V} \, 2 \, \cos(r s^1_z),
$$
By adapting the proof of Proposition \ref{prop:conversion}  to this setting 
we deduce that such a  spin system corresponds to a multi-occupancy double-dimer model
which is similar to the one which has been introduced in  Section \ref{sect:randomlooprepresentation}.
The only difference is that
dimers have $2N$ possible `colours', 
which we denote by
$1, \ldots, 2N$,
and,
similar to the model which was introduced in  Section \ref{sect:randomlooprepresentation},
the number of dimers of odd colour $i$ \textit{equals} the number of dimers of colour $i+1$ at each vertex and for each $i$. 
In the special case 
$$
U(n_1, n_2, \ldots n_N) =
\begin{cases}
1  & \mbox{ if  $n_1 + \ldots + n_N \in \{0,1\} $,} \\
0  & \mbox{ otherwise,} 
\end{cases}
$$
for example, one obtains a generalization of the monomer double-dimer model in which each vertex is either isolated or it belongs to a sequence of dimers whose colours are 
$2i$, $2i-1$, $2i$, $2i-1$,  $\ldots, 2i-1$ 
or $2i-1$, $2i$, $2i-1$, $2i$, $	\ldots, 2i$
for some integer $i \in \{1, \ldots, N\}$. 
Equivalently, this model can be defined 
by assigning to each configuration in
$\Omega$, the
set of monomer double-dimer configurations introduced in Section \ref{sect:monomerdoubledimerresults},  the weight 
\begin{equation}
\label{eq:probddimerN}
\forall \omega \in \Omega \quad
\mathbb{P}_{G, \rho}(\omega)  := 
\frac{\rho^{|M|} N^{ \mathcal{L}(\omega)  } }{ \mathbb{Z}_{G, N,  \rho} },
\end{equation}
in place of (\ref{eq:probddimer}),
where $ \mathcal{L}(\omega)$ is the total number of loops of at least two steps in $\omega$,
and $\mathbb{Z}_{G, N,  \rho}$ is a normalizing constant,
giving (\ref{eq:probddimer}) in the special case $N=1$.
   
   \subsection{Proof of Propositions \ref{prop:conversion} and  \ref{prop:conversionlocaltime}}
   \label{sect:proofofprop}
The proofs are based on the next simple identity, namely for each  vertex $x \in V$,  
and  integers $k_1, n_1, k_2, n_2 \in \mathbb{Z}$  we have 
\begin{equation}\label{eq:cancelation}
\frac{1}{(2 \pi)^{2 \, |V|}}\int_{\Omega_s}  \boldsymbol{ds}   \ (\overline{ S^1_x (\boldsymbol{s})})^{k_1} (S^1_x (\boldsymbol{s}))^{n_1} 
 \ (\overline{ S^2_x (\boldsymbol{s})})^{k_2} (S^2_x (\boldsymbol{s}))^{n_2}=
 \delta_{k_1, n_1} \delta_{k_2, n_2},
\end{equation}
where $\delta_{k,n}$ equals one if $k=n$ and zero otherwise.
This  follows from the definition of the spin variable (\ref{eq:spinvariable}).
By using the definition of  $\boldsymbol{\gamma}(\boldsymbol{s})$
(see equation (\ref{eq:alternation})),  we obtain from the previous identity that
for any pair of vectors $u^i = (u^i_x)_{x \in V} \in \mathbb{Z}^{V}$,
with $i \in \{1,2\}$,  the following key identity holds
\begin{equation}\label{eq:cancelation2}
\begin{split}
\int_{\Omega_s} \boldsymbol{d s}   \, 
\boldsymbol{\gamma}(\boldsymbol{s} ) \prod_{x \in V} \prod_{i=1}^{2}  (S_x^i(\boldsymbol{s}) )^{{{u}^i_x}}  
& =    \frac{1}{(2 \pi)^{2 \, |V|}}   \sum\limits_{ \boldsymbol{n} \in { \mathbb{N}_0}^{V}   } 
\int_{\Omega_s} \boldsymbol{d s} 
 \prod_{x \in V} 
 U(n_x)   {\big ( S_x^1(\boldsymbol{s}) \big )  }^{{{u}^1_x}}  {\big ( \overline{  S_x^1(\boldsymbol{s})} \big )  }^{{{n}_x}}   { \big ( S_x^2(\boldsymbol{s}) \big )  }^{{{u}^2_x}}  { \big (  \overline{S_x^2{(\boldsymbol{s})} } \big ) }^{{{n}_x}}  \\
& =  \prod_{x \in V}  U( {{u}_x^1}) \,  \,  \delta_{{{u}^1_x}, {{u}^2_x}}
\end{split} 
\end{equation}
where we extended the definition of $U$  to the negative integers by setting $U(n) = 0$ for each $n <0$
(this implies that the previous quantity equals zero as long as there exists some $x$ 
and $i \in \{1,2\}$ with  $u_x^i < 0$. 
We are now ready to prove the propositions using the previous  identity. 
Let $f : \Omega_s \mapsto \mathbb{C}$ be an arbitrary function.
Recall the definition of $G=(V,E)$ , of the extended graph $G_{en} =( V_{en}, E_{en})$,  
of the set of edges incident to the source vertex,
$E_s$,   and of the set of edges incident to the ghost vertex, $E_g$.
   We observe that the factor  $e^{  \beta    H(  \boldsymbol{s})  }$
   in (\ref{eq:measure})
   can be seen as a product of exponential functions.  
Each such exponential function is associated to an edge    of the extended graph and to a colour. 
   We start by performing a Taylor expansion for each such exponential in the product obtaining
 \begin{multline}\label{eq:starting}
   \int_{ \Omega_s }
 \boldsymbol{d s}  \,  
 \boldsymbol{\gamma}( \boldsymbol{s})  \,   f( \boldsymbol{s})
\, e^{  \beta    H(  \boldsymbol{s})  } \,   
   =        \sum\limits_{m^1 \in \mathbb{N}_0^{E} }
         \prod_{e \in E}  \frac{\beta^{m^1_e}}{m^1_e!}
     \sum\limits_{ m^2 \in \mathbb{N}_0^{E}  } 
     \prod_{e \in E}  \frac{\beta^{m^2_e}}{m^2_e!}
    \sum\limits_{ m^g \in \mathbb{N}_0^{E_g}  } 
    \prod_{e \in E_g}  \frac{(\beta h)^{m^g_e}}{m^g_e!}\\
 \times  \int_{ \Omega_s }
 \boldsymbol{d s}  \, 
 \boldsymbol{\gamma}( \boldsymbol{s}) 
f( \boldsymbol{s})
 \prod_{x\in V}\Big (  ( S_x^1(\boldsymbol{s}))^r+ (\overline{S_x^1(\boldsymbol{s})})^r \Big )^{m_{\{x,g\}}^g}
\prod_{x \in V} {(S_x^1(\boldsymbol{s}))}^{\tilde{n}_x(m^1)   } {(S_x^2(\boldsymbol{s}))}^{\tilde{n}_x(m^2)   }
   \end{multline}   
   where we  used that $2 \cos(r s_x^1 ) =  (S_x^1(\boldsymbol{s}))^r + (\overline{S_x^1(\boldsymbol{s})})^r$,
   we defined  $ {\tilde{n}_x}(m^i) = \sum_{y  \sim x } m_{\{x,y\}}^i$,
   and $y \sim x$ means that $y$ and $x$ are neighbour in the original graph $G$
   (this neighbourhood relation, then, does not consider the edges incident to the ghost  or source vertex).
In the previous sum we interpret  $m_{e}^i$ as the number of dimers of colour $i$
on the edge $e$ and $m_{\{x,g\}}$ as number of dimers
   on the edge  $\{x,g\}$ (independently from their colour).
   We now use at each vertex $x \in V$ the following identity 
    \begin{align*}
    \sum\limits_{m_{\{x,g\}} = 0}^{\infty}  \frac{(\beta h)^{  m_{\{x,g\}}  }}{m_{\{x,g\}}!} \Big (  ( S_x^1)^r+ (\overline{S_x^1})^r \Big )^{m_{\{x,g\}}}  & = 
    \sum\limits_{ m_{\{x,g\}}^1 =0 }^{  \infty   }
    \sum\limits_{ m_{\{x,g\}}^2 =0 }^{  \infty  } 
     \frac{(\beta h)^{  m_{\{x,g\}}^1  }}{m^{1}_{\{x,g\}}! }
   \frac{(\beta h)^{  m_{\{x,g\}}^2 }}{m^{2}_{\{x,g\}}! } 
   (S_x^1)^{ r \,m^1_{\{x,g\}} }   {(\overline{  S_x^1 })}^{ r \, m_{\{x,g\}}^2 } \\
    & =   \sum\limits_{ m_{\{x,g\}}^1, m_{\{x,g\}}^2 \in r \mathbb{N}_0 }
     \frac{ (\beta h)^{m_{\{x,g\}}^1 /r }}{(m^{1}_{\{x,g\}}/r) ! }
   \frac{ (\beta h)^{m_{\{x,g\}}^2 /r }}{(m^{2}_{\{x,g\}}/r)! }
   (S_x^1)^{  \,m^1_{\{x,g\}} }   {(\overline{  S_x^1 })}^{ m_{\{x,g\}}^2 }.
   \end{align*}
   and interpret $m^1_{\{x,g\}} $
   and $m^2_{\{x,g\}}$
   as the number of dimers of type respectively $1$ and $2$ on the edge $\{x,g\}$.
   Using 
 the previous identity  we then obtain that  the RHS of
   (\ref{eq:starting}) equals
 \begin{multline}\label{eq:expression1}
   \sum\limits_{  \substack{ m = (m^1, m^2) \in (\mathbb{N}_0^{  E \cup E_g })^2  : \\ m_{\{x,g\}}^1, m_{\{x,g\}}^2 \in r \mathbb{N}_0 \forall x} }
   \prod_{i=1}^{2} 
        \prod_{e \in E}  \frac{\beta^{m^i_e}}{m^i_e!}
    \prod_{e \in E_g}  \frac{(\beta h)^{ (m^{i}_e/r)}}{ (m^{i}_e/r)!}  \\ \times 
 \int_{ \Omega_s}
 \boldsymbol{d s}  \, 
 \boldsymbol{\gamma({s})} 
   f(\boldsymbol{s})
   \prod_{x \in V} {(S_x^1(\boldsymbol{s}))}^{\tilde{n}_x(m^1)   + m_{\{x,g\}}^1   - m_{\{x,g\}}^2} {(S_x^2(\boldsymbol{s}))}^{\tilde{n}_x(m^2) }
   \end{multline}   
   We are now ready to conclude the proof of the two propositions.
   For a lighter notation we will omit the dependence of $S_x^i$ on $\boldsymbol{s}$.
   \begin{proof}[Proof of Proposition  \ref{prop:conversion}]
      We let  $u^1, u^2 \in \mathbb{Z}^V$ be any two vectors. 
We set 
   $
   f = \prod_{x \in V} \prod_{i=1}^{2}  (S_x^i)^{u_x^i}.
   $
Under this choice we have that the starting expression, the LHS of (\ref{eq:starting}),  equals 
     \begin{multline}\label{eq:expression12}
  \sum\limits_{  \substack{ m = (m^1, m^2) \in (\mathbb{N}_0^{  E \cup E_g })^2  : \\ m_{\{x,g\}}^1, m_{\{x,g\}}^2 \in r \mathbb{N}_0 \forall x} }
   \prod_{i=1}^{2} 
        \prod_{e \in E}  \frac{\beta^{m^i_e}}{m^i_e!}
    \prod_{e \in E_g}  \frac{(\beta h)^{ (m^{i}_e /r)}}{ (m^{i}_e /r)!} \\ \times 
 \int_{ \Omega_s}
 \boldsymbol{d s}  \, 
 \boldsymbol{\gamma({s})} 
   \prod_{x \in V}  {(S_x^1)}^{u_x^1 + {\tilde{n}_x}(m^1) + m^1_{\{x,g\}} - m^2_{\{x,g\}}}
   {(S_x^2)}^{u_x^2+ {\tilde{n}_x}(m^2)  }
   \end{multline}
From   (\ref{eq:cancelation2})
we deduce that, for each $m = (m^1, m^2) \in \mathbb{N}_0^{  E \cup E_g }$ in the previous sum,
\begin{equation}\label{eq:keyrelation2}
\begin{split}
\int_{ \Omega_s}
 \boldsymbol{d s}  \, 
 \boldsymbol{\gamma({s})}  &
\prod_{x \in V}  {(S_x^1)}^{u_x^1 + {\tilde{n}_x}(m^1) + m^1_{\{x,g\}} - m^2_{\{x,g\}}}
   {(S_x^2)}^{u_x^2+ {\tilde{n}_x}(m^2)  }
 \\   =   
 & \prod_{x \in V} U( u_x^1 + {\tilde{n}_x}(m^1) + m^1_{\{x,g\}} - m^2_{\{x,g\}}  ) \, 
  \delta_{  u_x^1 + {\tilde{n}_x}(m^1) + m^1_{\{x,g\}} - m^2_{\{x,g\}},       u_x^2 + {\tilde{n}_x}(m^2)} \\
  \\  =   
 &  \prod_{x \in V} U( u_x^1 + {\tilde{n}_x}(m^1) + m^1_{\{x,g\}} - m^2_{\{x,g\}}  ) \, 
  \delta_{  u_x^{1,+} - u^{2,-}_x + {\tilde{n}_x}(m^1) + m^1_{\{x,g\}},      u_x^{2,+}  - u_x^{1,-} + {\tilde{n}_x}(m^2) + m^2_{\{x,g\}}},
\end{split}
  \end{equation} 
  where the components $u_x^{i, \pm}$ have been defined in the statement of the proposition. 
We see that, by (\ref{eq:keyrelation2}),
only the vectors $m = (m^1, m^2) \in \mathbb{N}_0^{ E \cup E_g }$
such that 
$u_x^{1,+} - u^{2,-}_x + {\tilde{n}_x}(m^1) + m^1_{\{x,g\}} = u_x^{2,+}  - u_x^{1,-} + {\tilde{n}_x}(m^2) + m^2_{\{x,g\}}$
at each vertex $x \in V$ are allowed to have non zero weight in the sum (\ref{eq:expression12}).
We can then associate to each such vector an element $m^\prime$ in the dimer cardinality set $\Sigma_r$,
which is defined in Definition \ref{def:dimercard},
such that ${m^\prime}^i_e = m^i_e$
for each $e \in E \cup E_g$ and each $i \in \{1,2\}$ 
and such that ${{m^\prime}^i_{\{x,s\}}}= u_x^{i,+}  - u^{i + 1 , -  }$
for each $x \in V$ and $i \in \{1,2\}$
(with the convention that \( i+1 \) is understood modulo~2, 
that is, \( i+1 = 2 \) if \( i = 1 \) and \( i+1 = 1 \) if \( i = 2 \)),
so that 
$$
U( u_x^1 + {\tilde{n}_x}(m^1) + m^1_{\{x,g\}} - m^2_{\{x,g\}} ) =
U(   u_x^{1,-} + u_x^{2,-} +   {\tilde{n}_x}({m^1}^\prime) + {m^1}^\prime_{\{x,g\}} - {m^2}^\prime_{\{x,g\}}  )
$$
Still denoting $m'$ by $m$,  the expression 
(\ref{eq:expression12}) can then be expressed as a sum over $\Sigma_r$ and equals then the following expression
  \begin{equation}\label{eq:expression2}
   \sum\limits_{  \substack{ m  \in \Sigma_r \,   :   \\  m_{\{x,s\}}^i = u_x^{i,+}  - u_x^{i + 1 , -  } \,  \, \forall x \in V, \, \,  \forall i \in \{1,2\}  }}
   \prod_{i=1}^{2} 
        \prod_{e \in E}  \frac{\beta^{m^i_e}}{m^i_e!}
    \prod_{e \in E_g}  \frac{(\beta h)^{m^{i}_e / r}}{(m^{i}_e/r)!}
 \prod_{x \in V} U(  {\tilde{n}_x}(m^1) + m^1_{\{x,g\}} - m^2_{\{x,g\}}   + u_x^{1,-}  + u_x^{2,-}  ).
   \end{equation}
This corresponds to the numerator in the RHS of (\ref{eq:correlation}) 
and concludes the proof. 
\end{proof}

\begin{proof}[Proof of Proposition  \ref{prop:conversionlocaltime}]
We   set  
$
f = \prod_{z \in A} \frac{ U(0)\, e^{- 2 h \beta \cos ( r s_z^1)} }{(2 \pi)^2 \,  \gamma_z(\boldsymbol{s})}
$
and obtain from (\ref{eq:expression1}) that the starting expression,
the LHS of (\ref{eq:starting}),  under this choice of $f$ equals 
\begin{multline}\label{eq:intermediate}
   \sum\limits_{  \substack{ m = (m^1, m^2) \in \mathbb{N}_0^{  E \cup E_g }  : \\ m^1_{\{x,g\}} = m^2_{\{x,g\}} = 0 \, \,  \forall x \in A \\ m_{\{x,g\}}^1, m_{\{x,g\}}^2 \in r \mathbb{N}_0 \, \, \forall x \in V}}
   \prod_{i=1}^{2} 
        \prod_{e \in E}  \frac{\beta^{m^i_e}}{m^i_e!}
    \prod_{e \in E_g}  \frac{(\beta h)^{m^{i}_e / r}}{(m^{i}_e/r)!}
    \\ \times 
 \int_{ \Omega_s}
 \boldsymbol{d s}  \, 
\big ( \frac{U(0)}{(2 \pi)^2} \big ) ^{|A|} 
  \prod_{z \in V \setminus A}  \, 
 {\gamma_z( \boldsymbol{s})} \, 
   \prod_{x \in V}  {(S_x^1)}^{{\tilde{n}_x}(m^1) + m^1_{\{x,g\}} - m^2_{\{x,g\}}} {(S_x^2)}^{{\tilde{n}_x}(m^2)  }
   \\ = 
     \sum\limits_{  \substack{ m = (m^1, m^2) \in \mathbb{N}_0^{  E \cup E_g }  : \\ m^1_{\{x,g\}} = m^2_{\{x,g\}} = \tilde n_x(m^1) = \tilde n_x(m^2) = 0 \, \, \forall x \in A \\ m_{\{x,g\}}^1, m_{\{x,g\}}^2 \in r \mathbb{N}_0 \, \,  \forall x \in V}}
   \prod_{i=1}^{2} 
        \prod_{e \in E}  \frac{\beta^{m^i_e}}{m^i_e!}
    \prod_{e \in E_g}  \frac{(\beta h)^{m^{i}_e / r}}{(m^{i}_e/r)!}
    \\ \times 
 \int_{ \Omega_s}
 \boldsymbol{d s}  \,  \prod_{z \in V}  \, 
 {\gamma_z( \boldsymbol{s})} \, 
   \prod_{x \in V}  {(S_x^1)}^{{\tilde{n}_x}(m^1) + m^1_{\{x,g\}} - m^2_{\{x,g\}}} {(S_x^2)}^{{\tilde{n}_x}(m^2)  }
   \\ = 
   \sum\limits_{  \substack{ m  \in \Sigma_r \,   :   \\  m^1_{\{x,g\}} = m^2_{\{x,g\}} = {\tilde{n}_x}(m^1) = {\tilde{n}_x}(m^2) =  0 \,  \, \forall x \in A    \\ 
   m^1_{\{x,s\}} = m^2_{\{x,s\}} = 0 \, \,  \forall x \in  V}}
   \prod_{i=1}^{2} 
        \prod_{e \in E}  \frac{\beta^{m^i_e}}{m^i_e!}
    \prod_{e \in E_g}  \frac{(\beta h)^{m^{i}_e / r}}{(m^{i}_e/r)!}
 \prod_{x \in V} U(  {\tilde{n}_x}(m^1) + m^1_{\{x,g\}} - m^2_{\{x,g\}}   ),
   \end{multline}
   where for the first identity we applied (\ref{eq:cancelation}) at each vertex $x \in A$ to observe that the  integral equals zero if there exists some edge $\{x,y\}$ with $x \in A$ and $y \sim x$ or $y = g$
   with $m_{\{x,y\}}^i >0$
   for some $i \in \{1,2\}$, 
    for the second step we used   (\ref{eq:cancelation2}) at each vertex $x \in V \setminus A$,  in analogy to 
   (\ref{eq:keyrelation2}),  obtaining the analogous of (\ref{eq:expression2}).
We then obtained the 
numerator in the RHS of  the first identity in Proposition \ref{prop:conversionlocaltime},
thus concluding the proof. 
   
   We now move to the proof of the second claim in the proposition by setting for $p \in \mathbb{N}_0$, 
   $
f =  \prod_{z \in A} 
    \frac{ \overline{ ( S_z^1 S_z^2 ) }^p }{(2 \pi)^2 \gamma_z(\boldsymbol{s})}.
   $
We obtain from (\ref{eq:expression1}) that, under this choice of $f$,  (\ref{eq:starting}) equals 
\begin{multline}\label{eq:expressionfinal}
  \sum\limits_{  \substack{ m = (m^1, m^2) \in \mathbb{N}_0^{  E \cup E_g }  :  \\ m_{\{x,g\}}^1, m_{\{x,g\}}^2 \in r \mathbb{N}_0 \forall x \in V}}
   \prod_{i=1}^{2} 
        \prod_{e \in E}  \frac{\beta^{m^i_e}}{m^i_e!}
    \prod_{e \in E_g}  \frac{(\beta h)^{ m^{i}_e/r}}{(m^{i}_e/r)!} \\
    \times 
 \int_{ \Omega_s}
 \boldsymbol{d s}  \,  
 \Big (  \prod_{z \in V \setminus A }  \, 
 {\gamma_z( \boldsymbol{s})} \Big ) 
  { \Big (  \frac{U(p)}{{(2 \pi)}^2} \Big ) }^{|A|}
   \prod_{x \in V  }  {(S_x^1)}^{{\tilde{n}_x}(m^1) + m^1_{\{x,g\}} - m^2_{\{x,g\}} - p \mathbbm{1}_{A}(x)} {(S_x^2)}^{{\tilde{n}_x}(m^2)  -p  \mathbbm{1}_{A}(x) } 
\end{multline}
where $\mathbbm{1}_{A}(x) =1$ if $x \in A$ and $\mathbbm{1}_{A}(x) =0$ otherwise. 
By using   (\ref{eq:cancelation2})
we deduce that, for each $m = (m^1, m^2) \in \mathbb{N}_0^{  E \cup E_g }$ in the previous sum,
\begin{multline}
\int_{ \Omega_s}
 \boldsymbol{d s}  \, 
 \Big (  \prod_{z \in V \setminus A }  \, 
 {\gamma_z( \boldsymbol{s})} \Big ) 
  { \Big ( \frac{U(p)}{(2 \pi)^2} \Big ) }^{|A|}
   \prod_{x \in V  }  {(S_x^1)}^{{\tilde{n}_x}(m^1) + m^1_{\{x,g\}} - m^2_{\{x,g\}} - p \mathbbm{1}_{A}(x)} {(S_x^2)}^{{\tilde{n}_x}(m^2)  -p  \mathbbm{1}_{A}(x) }
   \\ = 
    \prod_{x \in V \setminus A} U(  {\tilde{n}_x}(m^1) + m^1_{\{x,g\}} - m^2_{\{x,g\}}  ) \, 
  \delta_{   {\tilde{n}_x}(m^1) + m^1_{\{x,g\}},    {\tilde{n}_x}(m^2) + m^2_{\{x,g\}}} \\ 
 \times   \prod_{x \in A}  U( p ) \, \, 
  \delta_{  {\tilde{n}_x}(m^1) + m^1_{\{x,g\}} ,  p} \, \, 
  \delta_{   {\tilde{n}_x}(m^2) + m^2_{\{x,g\}}, p }.
\end{multline}
The proof now follows the same lines as the proof of Proposition \ref{prop:conversion} after replacing the previous expression in (\ref{eq:expressionfinal}).
\end{proof}

\section*{Acknowledgements} 
 L.T.  thanks Volker Betz and Benjamin Lees for sharing  discussions on a closely  related complex spin representation. 
 The authors thank Tom Spencer for illuminating comments
 and Alexandra Quitmann for producing the figure.
L.T.  thank the German Research Foundation (project number 444084038, priority program SPP2265) for financial support.  W.W was partially supported by MOST grant 2021YFA1002700,  NSFC grant 20220903 NYTP and SMEC grant 2010000080.  

\section*{Conflict of interest and data availability statement} 
The authors have no relevant financial or non-financial interests to disclose.  The manuscript has no associated data.

\end{document}